\documentclass[a4paper,11pt]{article}

\usepackage{geometry} 
\usepackage{amsmath}
\usepackage{graphicx}
\usepackage{amssymb}
\usepackage{epstopdf}
\usepackage{float}
\usepackage{caption}
\usepackage{hyperref}
\usepackage{mathtools}
\usepackage{color}
\usepackage{listings}
\usepackage{nicefrac}

\usepackage[style=numeric, isbn=false, url=false, doi=true, maxnames=10, backend=biber, sortcites=true]{biblatex}
\renewbibmacro{in:}{}
\DeclareFieldFormat
  [article,inbook,incollection,inproceedings,patent,thesis,unpublished]
  {title}{#1\isdot}
\usepackage[bblfile=targeted_immunisation]{biblatex-readbbl}

\renewcommand\labelenumi{(\roman{enumi})}
\renewcommand\theenumi\labelenumi

\usepackage[normalem]{ulem}
\usepackage{ mathrsfs }
\usepackage{bbm}
\usepackage[dvipsnames]{xcolor}
\usepackage{enumitem}

\usepackage{todonotes}

\usepackage{amsthm}

\usepackage[nottoc,notlof]{tocbibind}

\renewcommand{\labelenumi}{(\roman{enumi})}

\newtheorem{theorem}{Theorem}[section]
\newtheorem{corollary}[theorem]{Corollary}
\newtheorem{proposition}[theorem]{Proposition}

\newtheorem{remark}[theorem]{Remark}
\newtheorem{lemma}[theorem]{Lemma}
\newtheorem{claim}[theorem]{Claim}
\newtheorem*{claim*}{Claim}
\newtheorem*{theorem*}{Theorem}
\newtheorem*{proposition*}{Proposition}

\theoremstyle{definition}
\newtheorem{definition}[theorem]{Definition}

\usepackage{ wasysym }
\usepackage{ textcomp }

\newcommand{\1}{\mathbbm{1}}

\newcommand{\cA}{\mathcal{A}}
\newcommand{\cB}{\mathcal{B}}
\newcommand{\cC}{\mathcal{C}}

\newcommand{\cE}{\mathcal{E}}
\newcommand{\cF}{\mathcal{F}}
\newcommand{\cG}{\mathcal{G}}

\newcommand{\cT}{\mathcal{T}}

\newcommand{\cV}{\mathcal{V}}

\newcommand{\cL}{\mathcal{L}}

\newcommand{\eps}{\varepsilon}

\newcommand{\B}{\mathcal{B}}
\newcommand{\D}{\mathbbm{d}}
\newcommand{\cq}{c_X}

\newcommand{\de}{\operatorname{d}}

\newcommand{\sT}{\mathscr{T}}

\DeclareMathOperator{\E}{\mathbb{E}}
\DeclareMathOperator{\p}{\mathbb{P}}

\DeclareMathOperator{\Var}{Var}
\renewcommand{\P}{\mathbb{P}}
\newcommand{\N}{\mathbb{N}}
\newcommand{\Z}{\mathbb{Z}}

\newcommand{\nbig}{n_{\rm big}}

\newcommand{\T}{\mathbbm{t}}
\newcommand{\s}{\mathbbm{s}}

\title{Targeted Immunisation Thresholds for the Contact Process on Power-Law Trees}
\author{John Fernley, Emmanuel Jacob}
\date{\today}

\usepackage{framed}
\usepackage{float}
\usepackage[caption = false]{subfig}
\usepackage{graphicx}
\usepackage{makecell}

\usepackage{appendix}

\begin{document}

\begin{center}
{\huge  Targeted immunisation thresholds for the

 contact process on power-law trees}

\vspace{1.5em}

John Fernley\footnote{Corresponding author.}\textsuperscript{,}\renewcommand*{\thefootnote}{\fnsymbol{footnote}}\setcounter{footnote}{0}\footnote{HUN-REN 
R\'enyi Alfr\'ed Matematikai Kutat\'oint\'ezet,
			Re\'altanoda utca 13-15,
            1053 Budapest,
            Hungary. {\tt fernley@renyi.hu}}\qquad
Emmanuel Jacob\footnote{\'Ecole Normale Sup\'erieure de Lyon,
            46 all\'ee d'Italie,
            69007 Lyon,
            France. {\tt emmanuel.jacob@ens-lyon.fr}}

\vspace{1.5em}

\today

\end{center}

\renewcommand*{\thefootnote}{\arabic{footnote}}
\setcounter{footnote}{1}

\begin{abstract}
Scale-free configuration models are intimately connected to power law Galton-Watson trees. It is known that contact process epidemics can propagate on these trees and therefore these networks with arbitrarily small infection rate, and this continues to be true after uniformly immunising a small positive proportion of vertices. So, we instead immunise those with largest degree: above a threshold for the maximum permitted degree, we discover the epidemic with immunisation has survival probability similar to without, by duality corresponding to comparable metastable density. With maximal degree below a threshold on the same order, this survival probability is severely reduced or zero.

{\small
\vspace{1em}


}
\end{abstract}

\section{Introduction}

It is widely accepted that super-spreaders, namely particular individuals that have a greater potential for disease transmission, play a central role in the propagation of the COVID epidemic. More generally, a virus or an infection living on an heterogeneous population can use these heterogeneities to spread efficiently, even though a typical vertex can have a relatively small infection rate or effective reproduction number. Detecting and targeting these inhomogeneities can thus play a central role in the development of an efficient immunization strategy.

We aim to illustrate and study this effect on a simple idealized model. We model the epidemic by a contact process (or SIS infection), and consider a population given by a Galton-Watson tree with power-law offspring distribution, where immunization is taken into account simply by truncating the degree distribution. 
We first introduce this model and state our main results, and then we come back to the original motivation with a further discussion on the choice of the model.

On a locally finite graph $G=(V,E)$ modelling individuals and their interactions, the contact process is a Markov process in continuous time $(\xi_t)_{t\ge 0}$ on $\{0,1\}^V$ with $\xi_t(v)=1$ (resp. 0) if the individual $x$ is infected (resp. healthy) at time $t$. The dynamics are determined by each infected vertex recovering at rate $1$ and each edge $x\sim y$ with for example $x$ infected and $y$ healthy transmitting the infection (to $y$) at fixed rate $\lambda>0$. On an infinite rooted graph $(G,o)$, we will typically start the infection from only the root infected and look at the survival probability of the infection $P_{\rm surv}(\lambda):=\p(\forall t\ge 0, \xi_t\ne \mathbf{0})$. This survival probability is monotonic in $\lambda$, so that the survival probability is positive if $\lambda$ is larger than a critical parameter $\lambda_c=\lambda_c(G)\in [0,+\infty]$ and zero if $\lambda<\lambda_c$.
Note the result \cite{liggett1978} gives $\lambda_c(\Z^d)\leq \nicefrac{2}{d}$, and we have $\lambda_c\le \lambda_c(\N)<+\infty$ on any infinite graph by comparison to oriented percolation \cite{zbMATH02246274}. 

It is easy to prove that $\lambda_c\ge \nicefrac{1}{d}>0$ on any graph with degrees bounded by $d$, however it is totally possible to have $\lambda_c=0$ on graphs with unbounded degrees. 
We exactly know when this happens 
in the case of Galton-Watson trees 
since the works of \cite{huang2018contact} and \cite{bhamidi2019survival}. Consider $G$ a Galton-Watson tree with supercritical offspring distribution $X$ (namely $\E X>1$), and conditioned to be infinite (otherwise $\lambda_c(G)=\infty$). Then $\lambda_c(G)$ is:
\begin{itemize}
	\item positive a.s. if $X$ has an exponential tail, i.e. there exists $c>0$ such that $\E e^{cX}<+\infty$.
	\item 0 a.s. otherwise.
\end{itemize}
In the latter case, for any positive $\lambda$, the quenched survival probability $P_{\rm surv}(\lambda,G):=\p(\forall t\ge 0, \xi_t\ne \mathbf{0} | G)$ is thus a.s. positive on the event that $G$ is infinite, and in particular the annealed survival probability $P_{\rm surv}(\lambda)=\E[P_{\rm surv}(\lambda,G)]$ is also positive. 

We now specify the model and consider a Galton-Watson tree $G$ which:
\begin{itemize}
	\item is still supercritical: $\E[X]>1$.
	\item follows a power-law offspring distribution, or more precisely 
	\begin{equation}\label{tail_distrib}
	\P(X\ge k)\sim \cq k^{2-\tau}
	\end{equation}
	for some $\cq>0$ and power-law exponent $\tau>2$.
\end{itemize} 
Such a distribution has no finite exponential moment, so the (annealed) survival probability $P_{\rm surv}(\lambda)$ is positive for any $\lambda>0$. However this probability decays to 0 as $\lambda\to 0$, and as an adaptation %
of Proposition 1.4 in \cite{mountford2013metastable}, we can provide the following precise asymptotics for $P_{\rm surv}(\lambda)$:

\begin{equation}\label{Untruncated_Tree_Asymptotics}
P_{\rm surv}(\lambda)=
\begin{cases}
\Theta\Big(\lambda^{\frac {\tau-2}{3-\tau}}\Big), &\quad \text{when } \tau\le 2.5,\\
\Theta\Big(\frac {\lambda^{ 2\tau-4}}{\log^{\tau-2}(\nicefrac{1}{\lambda})}\Big), &\quad \text{when } 2.5<\tau\le 3,\\
\Theta\Big(\frac{\lambda^{2\tau-4}}{\log^{2\tau-4}(\nicefrac{1}{\lambda})}\Big), &\quad \text{when } \tau> 3,
\end{cases}
\end{equation}
where $\Theta$ stands for the usual Landau notation.

The goal of this paper is to see to what extent this result is affected when we remove the full descendants of the highest degree vertices, and thus replace the offspring distribution $X$ by $X \1_{X\le \T}$ for some truncation level $\T=\T(\lambda)$. 
 Note that we immediately obtain extinction of the contact process when $\T\le \nicefrac{1}{\lambda}$, and so  we will typically be interested in the regime $\T=\T(\lambda)> \nicefrac{1}{\lambda}$.

Our main result is as follows:

\begin{theorem*}
	Consider a truncation threshold $\T=\T(\lambda)\to +\infty$ as $\lambda\to 0$. Write $P^{\T}_{\rm surv}(\lambda)$ for the survival probability of the contact process with infection parameter $\lambda>0$, starting from only the root infected and run on the  Galton-Watson tree with truncated offspring distribution $X \1_{X\le \T}$ %
	\begin{enumerate}
		\item Suppose $\tau < 2.5$. Then there exists  a critical constant $\rho_c>0$ such that whenever $\rho_1 < \rho_c < \rho_2$ we have for small $\lambda$
		\[
		P^\T_{\rm surv}(\lambda)
		=
		\begin{cases}
		0, & \T \leq \rho_1 \lambda^{-\frac{1}{3-\tau}}, \\
		\Theta\left(\lambda^\frac{\tau-2}{3-\tau}\right), & \T \geq \rho_2 \lambda^{-\frac{1}{3-\tau}}.
		\end{cases}
		\]
		\item Suppose $\tau = 2.5$. Then we have constants $\rho_1, \rho_2 > 0$ such that for small $\lambda$
		\[
		P^\T_{\rm surv}(\lambda)
		=
		\begin{cases}
		0, & \T \leq \rho_1 \lambda^{-2}, \\
		\Theta\left(\lambda\right), & \T \geq \rho_2 \lambda^{-2}.
		\end{cases}
		\]
		\item Suppose $\tau \in (2.5,3]$. Then for any $\delta>0$ we can find $\rho_1, \rho_2 > 0$ such that for small $\lambda$
		\[
		P^\T_{\rm surv}(\lambda)
		=
		\begin{cases}
		O\left(
		e^{-\lambda^{-2+\delta}}
		\right), & \T \leq \rho_1 \lambda^{-2} \log \frac{1}{\lambda}, \\
		\Theta\left(\frac{\lambda^{2\tau-4}}{\log^{\tau-2}\left(\nicefrac{1}{\lambda}\right)}\right), & \T \geq \rho_2 \lambda^{-2} \log \frac{1}{\lambda}.
		\end{cases}
		\]
		\item Suppose $\tau>3$. Then for any $C>0$ we can find $\rho_1, \rho_2 >0$ such that for small $\lambda$
		\[
		P^\T_{\rm surv}(\lambda)
		=
		\begin{cases}
		O\left(
		\lambda^C
		\right), & \T \leq \rho_1 \lambda^{-2} \log^2 \frac{1}{\lambda}, \\
		\Theta\left(\frac{\lambda^{2\tau-4}}{\log^{2\tau-4}\left(\nicefrac{1}{\lambda}\right)}\right), & \T \geq \rho_2 \lambda^{-2} \log^2 \frac{1}{\lambda}.
		\end{cases}
		\]
	\end{enumerate}  
\end{theorem*}

\begin{remark}
	These are asymptotic results as $\lambda\to 0$ and the truncation threshold has $\T=\T(\lambda) \to +\infty$. In particular, the truncated degree distribution $ X  \1_{X\le \T}$ %
	will have first moment larger than 1 and the truncated Galton-Watson tree will be supercritical, as is necessary for the contact process to have any chance to survive.
	More generally and even when not explicitly stated, we will generally work on the assumption that $\lambda>0$ is sufficiently small.
\end{remark}

\begin{remark}
	In the case $\tau<2.5$, if we also choose $\T=\rho \lambda^{-\frac 1 {3-\tau}}$, we thus obtain a phase transition in $\rho$ where:
	\begin{itemize}
		\item For $\rho$ smaller than the critical value $\rho_c$, the contact process gets extinct almost surely
		\item For $\rho$ larger than the critical value $\rho_c$, the truncation does not have any appreciable effect on the survival probability, which is at most changed by a multiplicative constant when $\lambda\to 0$ as compared to $P_{\rm surv}(\lambda)$.
	\end{itemize}
\end{remark}

\begin{remark}
	In the cases where $\tau \geq 2.5$, we believe a similar phase transition occurs, from almost sure extinction of the contact process, to a survival probability which is essentially the same as without truncation, and in line with~\eqref{Untruncated_Tree_Asymptotics}. 
	However, our results are not as complete in those cases. First, we do not reveal a single critical value for $\rho$, but rather have to suppose $\rho$ large to get our large $\T$ estimates, and $\rho$ small to obtain our small $\T$ upper bounds. Second, if $\tau > 2.5$, we cannot even prove extinction of the process for small $\rho$, and instead we only get that the survival probability is significantly decreased. This is due to the difficulty of proving extinction of the contact process under some very unlikely but still possible events, like the truncated tree having many vertices with maximal degree adjacent to the root. Using our current techniques, we still can prove almost sure extinction of the infection, but only if we truncate at a much smaller degree $\T=O(\lambda^{-2})$.
\end{remark}

\begin{remark}
	In the $\tau<2.5$ regime we are talking about trees of maximal degree $o(1/\lambda^2)$, and so the seminal result of \cite{pemantle1992contact} tells us that this survival is only weak, i.e. the infection will recover permanently in any finite neighbourhood. In this case, then, we expect another transition to strong (local) survival at larger orders of $\T$, but we will not look at this second transition in this article.
\end{remark}

\begin{remark}
	In the physics literature, these questions have been considered by \cite{pastor2002epidemic,pastor2002immunization,deszo2003halting} using mean-field heuristics. However, as previously observed by \cite{chatterjee2009contact} for the problem without immunisation, these heuristics struggle to accurately account for the effect of vertices of large degree. Targeted immunisation of the large degree vertices is therefore where they are most difficult to justify.
	
This has led to estimates \cite[Equation 2]{deszo2003halting} for general networks with $2<\tau\leq 3$ or \cite[Equation 28]{pastor2002immunization} for a preferential atttachment network with $\tau=3$, both of the form
\[
\T=e^{\Theta \left( \nicefrac{1}{\lambda} \right)}
\]
and so starkly disagreeing with the polynomial order $\T$ of our main theorem.
\end{remark}

These results can be of interest for their own sake, but we now come back to the original motivation of modelling an epidemic on a scale-free network after targeted immunization.

\subsubsection*{Truncation and immunization.} 

Consider immunization instead of truncation, and write $\widetilde P^{\T}_{\rm surv}(\lambda)$ for the survival probability of the contact process with infection parameter $\lambda>0$, starting from only the root infected, on the Galton-Watson tree with offspring distribution $X$%
, but where every vertex with offspring\footnote{We could also immunize every vertex with \emph{degree} larger than $\T+1$, instead of every vertex with offspring larger than $\T$. This would make a difference only for the root (if it has degree exactly $\T+1$), but would not change our results.} larger than $\T$ is \emph{immunized}. By definition the immunized vertices cannot get infected or transmit the infection, so these vertices and their offspring can just be removed from the tree. Note this is different from truncation, where only the offspring of the truncated vertices were removed. In principle the contact process could have a different behaviour on the truncated tree, as a truncated vertex can still become infected and reinfect its parent. However, truncation or immunization make no difference when looking at our results.

\begin{claim}
	Our main theorem holds unchanged if we consider immunization instead of truncation, namely with  $P^{\T}_{\rm surv}(\lambda)$ replaced by  $\widetilde P^{\T}_{\rm surv}(\lambda)$.
\end{claim}

To see this, on the one hand we can see that the tree obtained after removal of the immunized vertices is contained in the truncated tree, and thus cannot have  a larger survival probability. On the other hand, our proofs of survival never use the truncated vertices, and thus hold unchanged in the immunization settings. 

\begin{remark}
Let $X$ follow the common offspring distribution. As discussed previously, the truncated tree is simply a Galton-Watson tree with offspring distribution $X \1_{X\le \T}$. The tree obtained after removal of the immunized vertices is also a Galton-Watson tree, but with a more complicated offspring distribution, which is the main reason for us to consider truncation afterwards. Indeed, let $Y$ have the law of $X$ conditionally on $X \leq \T$.
Then each vertex has probability $p=\p(X\le \T)$ to be a small vertex (with offspring $\le \T$), and conditionally on being a small vertex, has offspring distribution $Y$, each child being again a small vertex with probability $p$. Hence, if we condition on the root to be a small vertex, then the tree obtained after removal of the big vertices is a Galton-Watson tree with mixed binomial offspring distribution 
$
\operatorname{Bin}\left(Y,p\right).
$
\end{remark}

\subsubsection*{Galton-Watson tree and scale-free networks}

A population where an infection can spread is naturally modelled by a large but finite scale-free network. 
In \cite{mountford2013metastable}, Mountford et al. actually study the contact process on such a finite scale-free network, namely the configuration model with power-law offspring distribution $p$ which further satisfies $p(\{0,1,2\})=0$.
Whatever the infection rate $\lambda>0$, these graphs have sufficiently many vertices with high degrees to maintain the infection in a metastable state for an exponentially long time (in the number of vertices of the graph). 
Their main results are then sharp estimates for the associated metastable density, when $\lambda$ tends to 0. 

By duality, this metastable density is also the probability that the infection survives for a long (but subexponential in the graph size) time when started from a single uniformly chosen infected vertex. This in turn converges to the survival probability of the contact process on the local limit of these graphs, started with only the root infected. This local limit of the configuration model is a well-known two-stage Galton-Watson tree, where the root has degree distribution $p$ but every other vertex has degree distribution given by the size-biased version of $p$, and thus offspring distribution $q$ given by $q(k)=(k+1) p(k+1)/\sum i p(i)$. Most of their work focuses on estimating the survival probability on this two-stage Galton-Watson tree. The obtained estimates on this two-stage Galton-Watson tree in turn transfer easily to the estimates~\eqref{Untruncated_Tree_Asymptotics} on the one-stage Galton-Watson tree with offspring distribution $q$.

We certainly could have taken the same route, and studied targeted immunization on the scale-free configuration model. We expect our results to generalize to these settings, using a combination of our techniques with theirs. 

Note that if we isolate every immunized vertex, we will of course obtain a disconnected graph, and thus the survival probability on the local limit of the whole graph will typically describe the metastable ``density'' of infected vertices in the whole graph and not in a given (even giant) component. This effect didn't have to be taken into account in~\cite{mountford2013metastable} as the assumption $p(\{0,1,2\})=0$ made the graph already connected with high probability.

\subsubsection*{Organisation of the paper}

In Section~\ref{sec:overview} we present heuristics that help understand our results and in particular why the truncation threshold varies for different values of $\tau$. We also introduce some important technical tools used in our proofs, in particular a process that can interpolate between the contact process and a branching random walk upper bound. 
In Section~\ref{sec_general_upper} we establish an upper bound for the survival probability of the infection applying to all $\tau>2$ and any truncation level. Then, Sections~\ref{sec:small_tau},~\ref{sec:moderate_tau} and~\ref{sec:large_tau} correspond respectively to the phases $\tau\le 2.5$, $\tau\in (2.5,3]$ and $\tau>3$. Each section is then divided into two parts: the first one deals with large truncation threshold and provides a matching lower bound for the survival probability, thus showing the immunization is essentially inefficient. Then the second part of each section deals with small truncation threshold and proves either almost sure extinction, or a much smaller upper bound for the survival probability, thus showing the efficiency of immunization. 
These second parts of each section are both the most technical part of our work, and in our opinion the most interesting ones.

Finally, in Appendix \ref{appendix} we examine the contact process' survival on stars which will be important in understanding the effect of vertices with degree close to the immunisation threshold. For this often-considered problem we show an improvement to \cite[Lemma~3.4]{nguyen2022subcritical} by removing a logarithmic factor.

\section{Overview}\label{sec:overview}

In this section we present the heuristics for lower bounding the infection, which are close to the later lower bound proofs and give some intuition for the infection behaviour in each $\tau$ case, as well as our main proof techniques for upper bounding the infection.

\subsection{The SIR Galton-Watson tree}\label{sec:sir_subtree}

The SIR process is simply the contact process with permanent recoveries, which is a stochastic lower bound \cite[Section I.1]{liggett1999stochastic} . On a tree with root $o$ and initial infection $\xi_0=\mathbbm{1}_o$, this simplifies the analysis in that infection can only move away from the root. In fact, simply give each vertex an $\operatorname{Exp}(1)$ recovery time and each vertex but the root an $\operatorname{Exp}(\lambda)$ infection time from its parent, and then remove every edge for which the infection time of the child is larger than the recovery time of the parent. 
The connected component of $o$ is then the subtree of vertices which receive the infection.

Hence, putting the above construction in a Galton-Watson tree produces a \emph{Galton-Watson subtree} which is the easiest lower bound for our problem.

\begin{definition}
Draw $D$ from some offspring distribution, and $I \sim \operatorname{Exp}(1)$. Then the offspring distribution of the SIR Galton-Watson tree is $Y$ with distribution
\[
Y \vert D,  I \sim \operatorname{Bin}\left(
D,
1-e^{-\lambda I}
\right).
\]
\end{definition}

\subsection{Local survival by reinfection from your neighbours, and survival strategies}

On the star graph containing a central vertex and $k=k(\lambda)$ leaves, with initially only the center infected,
the center will typically not be permanently recovered before time
\begin{equation}
T(\lambda, k)= e^{\Theta(\lambda^2 k)}.
\end{equation}
That is to say, they will typically see the infection at some time later than such a $T$.

This simple result dates back to \cite[Theorem 4.1]{pemantle1992contact} and has been rewritten or detailed several times since, see e.g. \cite[Lemma 5.2]{berger2005starsurvival} or \cite[Lemma 3.4]{nguyen2022subcritical}. See also~\cite[Lemma 4]{periodic2020} for a similar upper bound for the expected extinction time of the epidemics. We will use a version of this result stated for the total time the center is infected, see Lemma~\ref{lem:star_survival_center}, which is actually a refinement of~\cite[Lemma 3.4]{nguyen2022subcritical} of independent interest.

On a larger graph, $T(\lambda, \deg(x))$ can still be interpreted heuristically as the time that the infection can survive ``locally'' at a vertex $x$, using only direct reinfections from the neighbours of $x$. On the truncated tree, the maximal such value is of course $T(\lambda, \T+1)$,
which differs a lot in the different regimes considered in our main theorem. Writing $D$ for the offspring distribution of the tree environment:
\begin{enumerate}
	\item Case $\tau\le 2.5$: 
	Consider $\T=\rho \lambda^{-\frac 1 {3-\tau}}$ for some constant $\rho>0$. In the case $\tau<2.5$ we have $T(\lambda,\T+1)\to1$ as $\lambda\to 0$, which is an indication that the reinfections barely have any effect on the spread of the epidemics. We thus expect to survive as in an SIR process, namely if and only if we have $\lambda \E[D]>1$. When $\tau=2.5$ we now have $T(\lambda,\T+1)=e^{\Theta(\rho)}$ of constant order. We still expect survival by the SIR lower bound if $\lambda \E[D]>1$, while we expect extinction if we have both $\lambda\E[D]<1$ and $\rho$ sufficiently small.
	\item Case $2.5<\tau\le 3$. Consider $\T=\rho \lambda^{-2} \log\nicefrac{1}{\lambda}$. We now have $T(\lambda,\T+1)=\lambda^{-\Theta(\rho)}$. When $\rho$ is large, this time is larger than $\lambda^{-1}$ and thus sufficient to typically infect any given neighbour. We then ensure survival of the infection by observing that in that case the high-degree vertices (with degree of order $\T$) always form a supercritical Galton-Watson process on which the infection can spread.
	
	On the contrary, when $\rho$ is small, we infect any given neighbour with probability only $\lambda^{1-\Theta(\rho)}$, which is not much larger than $\lambda$. In this phase we always have $\lambda \E[D]\to 0$ as well as $\lambda^{1-\Theta(\rho)}\E[D]\to 0$ when $\rho$ is small. Therefore the strategy of maintaining the infection at each vertex $x$ for time $T(\lambda, \deg(x))$ before passing the infection to its children will not succeed, which is an indication that the contact process might die out.
	\item Case $\tau>3$. Consider $\T=\rho \lambda^{-2} \log^2 \frac 1 \lambda$. In that case we have \[
	T(\lambda,\T+1)=\lambda^{-\Theta(\rho)\log(\nicefrac{1}{\lambda})},\] a time which is sufficient to infect not only the direct neighbours, but also all nearby vertices up to distance $\Theta(\rho \log(\nicefrac{1}{\lambda}))$. It is thus natural to consider that two high-degree vertices are ``connected'' if they are separated by less than $\Theta(\rho \log(\nicefrac{1}{\lambda}))$ small-degree vertices, hence obtaining a new ``Galton-Watson tree of high degree vertices''.
	
	As $\tau>3$ the high-degree vertices are more rare in the original tree, and actually typically at distance $\Theta(\log(\nicefrac{1}{\lambda}))$ apart. Thus depending on the value of $\rho$, the new ``Galton-Watson tree of high degree vertices'' will either be supercritical or subcritical.
\end{enumerate}

Actually, it is not too difficult to turn these heuristics into proper proofs of the lower bounds on the survival probability for large truncation thresholds. It is however considerably more difficult to obtain the upper bounds for small truncation thresholds. We briefly describe below some of the techniques used in the proofs.

\subsection{Branching random walk and supermartingale}\label{sec:brw}

It is classical to interpret the contact process as a particle system, and say there is a particle at site $x$ if the site/individual $x$ is infected. In this interpretation, particles are killed at rate $1$ (by recovery of the individual), and give birth with rate $\lambda$ to a new particle at each neighbouring site... with the constraint that the number of particles at each site is bounded by 1 (so giving birth to a particle at an already occupied site has actually no effect). If we remove this constraint and allow multiple particles to coexist at each site, we then obtain a simple branching random walk upper bounding the contact process. This point-of-view is for example explored in detail in \cite[Part I.4]{liggett1999stochastic}. The advantage of this upper bound is of course that the branching random walk is often easier to study as the particles now evolve independently. In particular,
\begin{itemize}
	\item On a network with degrees strictly less than $\frac{1}{\lambda}$, the total number of infected particles is a subcritical branching process, whose expectation decays exponentially. On the contrary, if the degrees are strictly larger than $\frac{1}{\lambda}$, then the branching random walk is supercritical and has a positive probability of surviving forever (possibly only weakly/globally).
	\item On the star graph $S_k$ with one central vertex and $k$ leaves, the branching random walk is subcritical if $k< 1/\lambda^2$ but survives (with positive probability) if $k>1/\lambda^2$. 
	\item On a tree with degrees strictly less than $1+\frac{1}{4\lambda^2}$, the branching random walk upper bound does not survive locally \cite[Theorem 4.8(d)]{liggett1999stochastic}, and thus there is no local survival of the contact process (that is, after some finite time the root will stay healthy forever). On the contrary, on a tree with all degrees strictly larger than  $1+\frac{1}{4\lambda^2}$, there is strong/local survival.
\end{itemize}
Another relatively standard technique lies in defining supermartingales associated to the contact process, for example as in Lemma 5.1(i) of \cite{mountford2013metastable}, and in particular the process
\[
M^{y}_t:=\sum \xi_t(x) (2\lambda)^{d(x,y)},
\]
for any choice of a particular vertex $y$, on a tree with vertex degrees bounded by $1/8\lambda^2$. The two techniques are however closely related: on a tree with degrees bounded by $1/8\lambda^2$, an easy way to prove that the branching random walk mentioned above will not survive locally is just to associate a supermartingale similar to $M^y$, but with $\xi_t(x)$ replaced by the number of particles of the branching random walk located in $x$.

One interpretation of the fact that $M^y$ is a supermartingale is that starting from a given infected vertex $x$, you only have probability $2\lambda$ to succeed in infecting each vertex on the path from $x$ to $y$, which we could relate to the statement $T(\lambda, \nicefrac{1}{8\lambda^2})\le 2$. 
In the regime $\tau<2.5$ where $T(\lambda,\T+1)\to 1$, we will require a straightforward generalization of this supermartingale. 
\begin{lemma}\label{lemma_hostile_environment_set}
	On a tree $\cT$ with degrees bounded by $\nicefrac{\delta}{\lambda^2}$ with $\delta\le \nicefrac{1}{4}$, the processes
	\[
	M^{y}_t:=\sum \xi_t(x) \left(\frac {2\lambda}{1+\sqrt{1-4\delta}}\right)^{d(x,y)}
	\]
	are supermartingales. As a consequence, for any $x,y \in \cT$,
	\[
	\mathbb{P}\left(
	(x,0)\leftrightarrow\{y\} \times \mathbb{R}_+
	\right)
	\leq
	\left( 
	\frac{2\lambda}{1+\sqrt{1-4\delta}}	
	\right)^{\de(x,y)}.
	\]
\end{lemma}

\begin{proof}
Each infected vertex $x$ recovers at rate $1$, and so decreases $M_t^y$ at rate equal to its weight \[\left(\frac {2\lambda}{1+\sqrt{1-4\delta}}\right)^{d(x,y)}.\] Relative to this weight and by infection, it increases $M_t^y$ at rate bounded by
\[
\lambda \cdot \left( \frac{\delta}{\lambda^2} -1 \right) \cdot \frac{2\lambda}{1+\sqrt{1-4\delta}}
<
\frac{2\delta}{1+\sqrt{1-4\delta}}
\]
in its $\left( \nicefrac{\delta}{\lambda^2} -1 \right)$ children, and at rate bounded by
\[
\lambda \cdot \frac{1+\sqrt{1-4\delta}}{2\lambda}
=
\frac{1+\sqrt{1-4\delta}}{2}
\]
in its parent. Because
\[
\frac{2\delta}{1+\sqrt{1-4\delta}}+\frac{1+\sqrt{1-4\delta}}{2}=1
\]
this shows the supermartingale property.
\end{proof}

Note that $2/(1+\sqrt{1-4\delta})$ tends to 1 as $\delta\to 0$, which is in line with the statement $T(\lambda,\T+1)\to 1$ if $\T=o(\nicefrac{1}{\lambda^2})$. One may also compare this upper bound with the trivial lower bound
\[
\mathbb{P}\left(
(x,0)\leftrightarrow\{y\} \times \mathbb{R}_+
\right)
\geq
\left( 
\frac{\lambda}{1+\lambda}	
\right)^{\de(x,y)},
\]
to see that the BRW (branching random walk) and CP (contact process) hitting probabilities are approximately the same for small $\delta$.

These techniques are simple and useful when the degrees are bounded by $\nicefrac{1}{4\lambda^2}$ but cannot treat accurately vertices of larger degrees, as the branching random walk is supercritical if there is a single vertex of degree $k>\nicefrac{1}{\lambda^2}$, while we can still expect that the contact process will not survive around this vertex for time larger than $T(\lambda,k)=e^{\Theta(\lambda^2 k)}$.

To treat the vertices with higher degrees, we start with simple trees before going to more complicated ones:
\begin{itemize}
	\item On a star graph with $k$ leaves, the expected extinction time of the contact process is bounded by $e^{\Theta(\lambda^2 k)}$.
	\item On a tree with one big vertex of degree $k$ and other vertices of degrees bounded by $1/8\lambda^2$, %
	we can still provide an upper bound $e^{\Theta(\lambda^2 k)}$ for the total time that the center is infected, and deduce an upper bound for the hitting probabilities of the vertices in the tree. See Lemma 5.1 and Corollary 5.2.
	\item More complicated trees require more work, and in particular the introducing of another process upper bounding the contact process: the ``contact process--branching random walk'' that we now present.
\end{itemize}

\subsection{Contact process--Branching random walk}\label{sec:CP_BRW}

In order to tackle the case of larger degrees, we use the following simple refinement of the branching random walk upper bound, inspired by~\cite{periodic2020}, where the branching phenomenon is considered only on a specific subset of all the vertices. This upper bound is an intermediate process between the contact process and the branching random walk and shares some features with both, so we will denote it by ``contact process-branching random walk'', or CP-BRW. This process is best described as a ``multitype contact-process'', with the interpretation of the contact process as an interacting particle system as in the previous section.
\begin{definition}[continuous-time CP-BRW with sinks $\cV$]\label{def_continuous_cpbrw}
	Select a set of ``sinks''  $\cV \subseteq V\left(\cG\right)$. The continuous-time CP-BRW with sinks $\cV$ is a multitype contact process, where each type of infection:
	\begin{itemize}
		\item is initiated at some sink $v\in \cV$. 
		\item evolves as the usual contact process on $\left(V\left(\cG\right) \setminus V\right) \cup \{v\}$. Note this CP has thus at most one particle at each site, and can infect the sink $v$, but not the other sinks.
	\end{itemize}
	The different types of particles evolve independently, except for the following rule for the creation of new types of infections:
	\begin{itemize}
		\item For any sink $v\in \cV$, create a new type of infection, initiated at $v$ (and thus starting from only $v$ infected), at a rate given by $\lambda$ times the total number of particles present at the neighbouring sites $z\sim v$, and which have types initiated at other sinks $w\ne v$.
	\end{itemize}
\end{definition}
\begin{remark}
In this description, we stress that the type of an infection is important for two reasons. First, if the type is initiated at $v$, then the contact process evolves on the non-sink sites \emph{and also the site} $v$. But second and more importantly, its evolution is independent from all the previously created types, even if initiated at the same sites. We further stress that:
\begin{itemize}
	\item We typically start the process from one particle, initiated at some sink $v\in \cV$.
	\item A contact process of given type and initiated at a sink $v\in \cV$ cannot imply any creation of new types of infections initiated at $v$. However, it can imply the creation of more than one type of infections initiated at each sink $w\in \cV\setminus \{v\}$.
	\item To each creation of a new type of infection (say at sink $v$), we can identify the particle which led to this creation (necessarily present at a neighbouring site $z\sim v$), and its type. Referring to this type as ``the parent type'', we see that the different types of infections have a natural tree structure, rooted at the original type (the type of the first particle present at time 0).
	\item It is not difficult to see that the total number of particles present at any given site (disregarding their types) is a stochastic upper bound for the original contact process starting from only $v$ infected. This is best proven using the graphical construction of the contact process, as constructed and applied in \cite[III.6]{liggett1985interacting}. 
	Actually, any increase in the set of sinks means a further stochastic upper bound, and if we consider all sites to be sinks, we just obtain (a complicated description of) the continuous-time Branching random walk of  Section~\ref{sec:brw}. Heuristically, adding sinks allows to replace the dynamics of the contact process with the easier dynamics of a branching process, however with the possible drawback of obtaining too crude upper bounds if we introduce too many sinks.
\end{itemize}
\end{remark}
\begin{definition}[discrete-time CP-BRW with sinks $\cV$]\label{def_cpbrw}
	For all sinks $w\in \cV$ and $n\ge 0$, write $\phi^\cV_n(w)$ for the number of types of infections initiated at $w$ and with depth $n$ in this tree. Then $\left(\phi^\cV_n\right)_{n\in \mathbb{N}}$ is a discrete-time branching-random walk on $\cV$, which we will denote as the discrete-time CP-BRW with sinks $\cV$, or simply CP-BRW with sinks $\cV$.
\end{definition}
The discrete-time CP-BRW $\left(\phi^\cV_n\right)_{n\in \mathbb{N}}$ contains much less information than the continuous-time CP-BRW -- in particular the time-dynamics are totally lost -- but it will be sufficient for our purposes. 
From now on, if a type of infection initiated at some sink $v\in \cV$ has created $k$ types of infections initiated at some other sink $w\in \cV$ (namely their children in the tree of the types of the infections), we will simplify the terminology, and simply say that the infection initated at $v$ has ``sent $k$ infections'' to $w$. We will also write $m^{\cV}(v,w)$ for the mean number of infections sent to $w$ by a single infection at $v$, where $v$ need not necessarily be a sink, or simply $m(v,w)$ for $m^{\{w\}}(v,w)$ when there is no other sink. %

We will often want to bound the probability of the infection hitting a target sink $w$. To this end, it is natural to consider the CP-BRW \emph{with absorption at} $w$. It is defined as the CP-BRW in Definition~\ref{def_cpbrw}, except that each walker at $w$ now has the deterministic offpsring of a single walker at $w$. We will denote this process by $\phi^{\cV,[w]}$, and write $m_\infty^{\cV}(v,w):=
\E_{\mathbbm{1}_v}\left[\phi_\infty^{\cV,[w]}(w)\right]$ for the mean number of infections absorbed at the target sink $w$ when the process is started from a single infection at $v$. Note we can bound the probability that the contact process ever hits $w$ by $m(v,w)\le m_\infty^{\cV}(v,w)$. The branching structure of the CP-BRW, allows to express this upper bound as:
\begin{equation}\label{equation_branching_sum}
	m_\infty^{\cV}(v,w) = \sum_{\substack{x_0=v,\ x_k=w\\ x_1 \dots x_{k-1} \in \cV\setminus\{w\}\\  \forall i,\ x_i\ne x_{i-1}}} \prod_{i=1}^k  m^{\cV}(x_{i-1},x_i).
\end{equation}

\section{General upper bounds on the survival probability}\label{sec_general_upper}

We prove the upper bounds of~\eqref{Untruncated_Tree_Asymptotics}, which of course also imply the same upper bounds for the survival probability after truncation, and in particular those stated for large $\T$ in our main theorem. 

First, without loss of generality, we can suppose $q(\{0,1\})=0$ by adding 2 children to each vertex, which of course cannot decrease the survival probability. Consider the distribution $p$ on $\{3,4,\ldots\}$ given by $p(k)=k^{-1}q(k-1)/\sum i^{-1}q(i-1)$, so that we also have $q(k)=(k+1)p(k+1)/\sum i p(i)$. Letting $P_{\rm surv}^{(p,q)}(\lambda)$ denote the survival probability of the contact process starting from only the root infected on the two-stage Galton-Watson tree where the root has offspring distribution $p$ and other vertices have offspring distribution $q$, Proposition (1.4) in~\cite{mountford2013metastable} states
\begin{equation}
P_{\rm surv}^{(p,q)}(\lambda)=
\begin{cases}
O\left(\lambda^\frac{1}{3-\tau}\right) \quad \text{when }\tau \leq 2.5,\\
O\left(\frac{\lambda^{2\tau-3}}{\log^{\tau-2}\left(\nicefrac{1}{\lambda}\right)}\right)  \quad \text{when }\tau \in (2.5,3],\\
O\left(\frac{\lambda^{2\tau-3}}{\log^{2\tau-4}\left(\nicefrac{1}{\lambda}\right)}\right) \quad \text{when }\tau>3.
\end{cases}
\end{equation}
Writing $P_{\rm surv}^{(q)}(\lambda)=P_{\rm surv}(\lambda)$ for the survival probability of the contact process on the GW tree with offspring $q$, we also have
\[
P_{\rm surv}^{(p,q)}(\lambda)\ge \frac \lambda {1+\lambda} P_{\rm surv}^{(q)}(\lambda),
\]
as the contact process has probability $\nicefrac{\lambda}{(1+\lambda)}$ to infect a given neighbour before the recovery of the root, and starting from that point it has probability at least $P_{\rm surv}^{(q)}(\lambda)$ to survive in the subtree containing this neighbour (see also the inequality just after (4.11) in~\cite{mountford2013metastable} for a similar statement). The result follows immediately.

It would also be easy to adapt~\cite{mountford2013metastable} to obtain the lower bounds of~\eqref{Untruncated_Tree_Asymptotics}, for example the cases $\tau>2.5$ are almost immediately covered by their inequality (6.15). However, in next sections we prove the stronger result that these lower bounds still hold after truncation (for large truncation threshold), which cannot follow that easily from~\cite{mountford2013metastable}. In fact when $\tau >2.5$  they use a single path of degrees increasing to infinity, which clearly doesn't exist in the truncated tree. We thus replace this single path of vertices with degrees tending to infinity, by a branching process of high degree vertices (with still bounded maximal degree) which can maintain the infection forever, as already explained in previous section.

\section{\texorpdfstring{Power-law exponent $\tau\le 2.5$}{Power-law exponent t < 2.5}}\label{sec:small_tau}

In this section we address the parameter regime $\tau\leq 2.5$ of our main theorem by Propositions \ref{prop_small_tau_survival} and \ref{prop_small_tau_subcrit_extinction}.

\subsection{Supercritical truncation threshold}

For $\tau\in (2,2.5]$ it suffices to use the classical SIR lower bound of Section \ref{sec:sir_subtree}. 
In a Galton-Watson tree with offspring distribution $D$, the vertices ever infected form a Galton-Watson tree with offspring distribution $Y$, with mean
\begin{equation}\label{mean_offspring_Y}
\E[Y]= \frac {\lambda}{1+\lambda}\E[D].
\end{equation}

Recalling from the Introduction that the underlying offspring distribution has $q([k,+\infty))\sim \cq k^{2-\tau}$, we have the following result.

\begin{proposition}\label{prop_small_tau_survival}
	For $\tau\in (2,2.5]$, define the critical value $\rho_c$ as $\cq \frac {\tau-2}{3-\tau} \rho_c^{3-\tau}=1$ or equivalently
	\begin{equation}\label{def_rho_c}
	\rho_c=\left(\frac {3-\tau}{(\tau-2) \cq}\right)^{\frac{1}{3-\tau}},
	\end{equation}
	and suppose $\rho>\rho_c$ and $\T(\lambda)\ge \rho \lambda^{-\frac{1}{3-\tau}}$. Then we have $P_{\rm surv}^\T(\lambda)=\Omega\left(\lambda^\frac{\tau-2}{3-\tau}\right)$ as $\lambda \downarrow 0$.
\end{proposition}

\begin{remark}
	By a slight refinement of our proof that we do not detail, we also get that the contact process survives with probability $\Theta(1)$ when the root is a big vertex with degree $\Theta\left( \lambda^{-\frac{1}{3-\tau}}\right)$, an event which has probability $\Theta\left(\lambda^\frac{\tau-2}{3-\tau}\right)$.
\end{remark}
\begin{proof}
	The survival probability $P_{\rm surv}^\T(\lambda)$ is clearly nondecreasing in $\T$ so we assume $\T(\lambda)\sim \rho \lambda^{-\frac{1}{3-\tau}}$. 
	When $\lambda\to 0$ and $\T\to \infty$, the expectation of the truncated degree $D$ has
	\begin{equation}\label{truncated_degree_expectation}
	\E[D]=\sum_{k=1}^\T \p(D\ge k) \sim \int_1^{\T(\lambda)}\cq (x^{-(\tau-2)}-\T^{-(\tau-2)})\mathrm dx \sim \lambda^{-1} \cq \frac {\tau-2}{3-\tau} \rho^{3-\tau}, 
	\end{equation}
	and thus by~\eqref{mean_offspring_Y} the mean offspring of the Galton-Watson process has $\E[Y]\to m:=\cq \frac {\tau-2}{3-\tau} \rho^{3-\tau},$ which is larger than 1 by the assumption $\rho>\rho_c$, and thus the process is supercritical.  
	We further lower bound the survival probability by a simple second moment computation with
	the following classical lemma.
	\begin{lemma}\label{lem:survivalGW_by2ndmoment}
		For any Galton-Watson process with offspring distribution with mean $m>1$ and variance $\sigma^2$, the survival probability is bounded below by $\frac {m(m-1)}{\sigma^2+m(m-1)}.$
	\end{lemma}
To prove this lemma, it suffices to use \cite[Chapter 1 (8.4)]{harrisbranching} to see that $W:=\lim_{n \rightarrow \infty} m^{-n} Y_n$, namely the limit of the martingale associated with the GW process, has $\E[W]=1$ and
\[
\E[W^2]=1+\frac{\sigma^2}{m(m-1)},
\]
and conclude by the Paley-Zygmund inequality.
	We come back to the proof of Proposition~\ref{prop_small_tau_survival} and upper bound the variance of $Y$ with
	\begin{align*}
	\Var(Y)&=\Var(\E[Y\big| D,I])+ \E[\Var(Y|D,I)]\\
	&= \Var(D(1-e^{-\lambda I}))+\E[\Var(\operatorname{Bin}(D,1-e^{-\lambda I}))]\\
	&\le \E[D^2]\E[(1-e^{-\lambda I})^2]+ \E[D e^{-\lambda I} (1-e^{-\lambda I})]
	\end{align*}
	The second term in the sum is simply bounded by $\E[Y]$. For the first term, we simply compute $\E[(1-e^{-\lambda I})^2]=\frac {2\lambda^2}{(1+\lambda)(2+\lambda)}=O( \lambda^2)$ and $
	\E[D^2]
	\leq
	\E[D\T]
	=O(\lambda^{-\frac {4-\tau}{3-\tau}})
	$.
	It follows $\Var Y=O(\lambda^{-\frac {\tau-2}{3-\tau}})$ and thus by Lemma~\ref{lem:survivalGW_by2ndmoment} we get the result.
\end{proof}

\subsection{Subcritical truncation threshold}

\begin{proposition}\label{prop_small_tau_subcrit_extinction}
	For $\tau \in (2,2.5)$ and $\rho<\rho_c$, any truncation function
	\[
	\T(\lambda)\leq \rho \lambda^{-\frac{1}{3-\tau}}
	\]
	will have, given $\lambda>0$ sufficiently small, $P_{\rm{surv}}^\T(\lambda)=0$.
	
	For $\tau = 2.5$ and some $\rho$ with
	\[
	\rho\le \frac{1}{4}, \quad \frac{2 \cq \sqrt{\rho}}{1+\sqrt{1-4\rho}} <1
	\]
	any truncation function
	\[
	\T(\lambda)\leq \rho \lambda^{-2}
	\]
	will have, given $\lambda>0$ sufficiently small, $P_{\rm{surv}}^\T(\lambda)=0$.
\end{proposition}

\begin{proof}
	The infection process dies out a.s. if we can show that the probability of the infection reaching distance $k$ from the root tends to 0 as $k\to \infty$. To this end, we simply combine the mean number of vertices at depth $k$ given by~\eqref{truncated_degree_expectation} with the probability of infecting any given such vertex upper bounded by Lemma~\ref{lemma_hostile_environment_set}.
	
	For $\tau=2.5$, we take $\delta=\rho$ in Lemma~\ref{lemma_hostile_environment_set} to obtain
		\[
		\frac{2\lambda \E[D]}{1+\sqrt{1-4\rho}}\to \frac{2 \cq \sqrt{\rho}}{1+\sqrt{1-4\rho}}<1,
		\]
	which taken to the power $k$ indeed tends to 0.
	
	For $\tau<2.5$, the truncation threshold is $o(\lambda^{-2})$, and thus bounded by $\delta/\lambda^2$ with arbitrary $\delta>0$. We conclude in that case by observing
		\[
	\frac{2\lambda \E[D]}{1+\sqrt{1-4\delta}}\to \frac{2}{1+\sqrt{1-4\delta}} \cq \frac {\tau-2}{3-\tau} \rho^{3-\tau},
	\]
	which is less than 1 when $\rho<\rho_c$ and $\delta$ small.
\end{proof}

\section{\texorpdfstring{Power-law exponent $\tau\in(2.5,3]$}{Power-law exponent 2.5 < t < 3}}\label{sec:moderate_tau}

In this section we address the parameter regime $\tau\in(2.5,3]$ of our main theorem. First in Section \ref{ssec:large_thresh_moderate} we lower bound the survival probability in the supercritical case, and finally in Theorem \ref{theorem_moderate_conclusion} we find our subcritical upper bound.

\subsection{Large truncation threshold}\label{ssec:large_thresh_moderate}

When $2.5<\tau\le 3$ we ensure survival of the infection by using again a SIR-like lower bound, with the refinement that big vertices use their small children to maintain the infection locally for a long enough time so as to typically infect its big children.

More precisely, say a vertex is big if it has offspring larger than $\s=\mbox{$\frac{1}{\eps}$} \lambda^{-2}\log(\nicefrac{1}{\lambda})$, where $\eps$ is as in Lemma~\ref{lem:star_survival_center}. Note we will typically assume $\rho>\nicefrac{1}{\eps}$, so that the truncation threshold $\T=\rho \lambda^{-2}\log(\nicefrac{1}{\lambda})$ satisfies $\T\ge\s$, as is necessary to have any big vertex.

Suppose the root vertex $o$ is big, with degree $D_o\ge \s$, and consider the contact process evolving only on $o$ and its first $\s$ children (namely we consider here only infections going back and forth to these $\s$ children). If for this contact process the total time $o$ is infected is at least $\nicefrac 1 \lambda$, we say there is local survival at $o$, and write $L_o$ for this event. By Lemma~\ref{lem:star_survival_center}, the probability of $L_o$ is at least $1-\eps c/\log(\nicefrac{1}{\lambda})$, which is larger than $\nicefrac{1}{2}$ for small $\lambda$.

Moreover, conditionally on $L_o$, the vertex $o$ is infected for a time period (at least)  $\nicefrac{1}{\lambda}$, during which it communicates (at least) $\operatorname{Pois}(1)$ infections to each of its remaining $D_o-\s$ children, which are therefore independently infected with probability at least $1-e^{-1}\ge \nicefrac{1}{2}$. We then consider each of these infected children as a source of an independent infection living in its own branch for which we use the same survival strategy.

Doing so, we can ensure survival of the contact process if we have survival of a Galton-Watson process with an offspring distribution $Y$ that is constructed in three steps:
\begin{itemize}
	\item First, pick the offspring of the vertex, $D\sim X \1_{X\le \T}$.
	\item If $D\ge \s$, then pick $L\sim \operatorname{Ber}(\nicefrac 1 2)$, where $L=1$ stems for local survival at the vertex.
	\item Finally, pick $Y\big| D, L \sim \operatorname{Bin}\left(\1_{D\ge \s} L.(D-\s),\frac 12\right),$ the number of the remaining $D-\s$ children of the vertex to which the vertex communicates an infection.
\end{itemize}

As for the case $\tau\le 2.5$, we lower bound the survival probability by a second moment method on $Y$ using Lemma~\ref{lem:survivalGW_by2ndmoment}. We have if $\tau<3$
\[
\E[Y]=\frac 1 2 \E[\1_{D\ge \s}L.(D-\s)]
=\frac 1 4 \E[(D-\s) \1_{D\ge \s}]
=\Omega(\E[D])
=\Omega(\T^{3-\tau})
\]
while in the case $\tau=3$:%
\[
\E[Y]=\frac 1 4 \E[(D-\s) \1_{D\ge \s}]
= \frac 1 4  \int_\s^\T \p(D \geq z) {\rm d}z
\sim \frac{\cq}{4}
\left(
\log(\eps \rho)
-1+\frac{1}{\eps \rho}
\right)
>1
\]
so that the final inequality is true whenever $\rho$ is large enough.

The variance of $Y$ has
\begin{align*}
\Var(Y)&= \E\left[\Var(Y \big| D, L)\right] + \Var\left(\E[Y \big| D, L]\right) \\
&\le \E\left[\frac {1} 4 \1_{D\ge \s}L.(D-\s)\right] 
+ 
\Var\left(\frac {1} 2 \1_{D\ge \s}L.(D-\s) \right) \\
&\le \frac {\E[Y]}2+\frac {\T-\s}2\E[Y],
\end{align*}
where in the second term we bound the variance by the expectation of the square of the random variable and use that the random variable is bounded by $(\T-\s)/2$. By Lemma~\ref{lem:survivalGW_by2ndmoment} we then obtain a lower bound for the survival probability of order
\[\frac {\E[Y]} \T=\Theta(\T^{2-\tau})=\Theta\left(\frac {\lambda^{2\tau-4}}{\log^{\tau-2}\frac 1 \lambda}\right).\]

\subsection{Small truncation threshold}\label{sec:small_truncation_threshold_in_2.5-3}

We still consider the truncation threshold 
$\T= \rho \lambda^{-2}\log \frac{1}{\lambda}$, but now with a smaller value of $\rho$ (in particular with $\rho$ smaller than $1/\eps$ from last section), so heuristically a vertex will maintain the infection locally at most up to time 
\[
e^{\Theta(\lambda^2 \T)}= \lambda^{-\Theta(\rho)},
\]
which, although being large, will be smaller than $\lambda^{-1}$ by the choice of $\rho$. 
Thus even a vertex of maximal degree has only small probability of ever infecting any given direct neighbour, which one might see as a first indication that getting upper bounds in that case might be feasible. 
To provide a rigorous upper-bound, though, we will:
\begin{itemize}
	\item First consider the case of a single vertex of degree larger than $1/8\lambda^2$. In that case, we 
\end{itemize}
 use the Contact process-Branching random walk of Section~\ref{sec:CP_BRW}, with sinks all ``big vertices'', namely all vertices with degree larger than $1/8\lambda^2$.

As

\begin{lemma}\label{lemma_expected_infectious_period_general}
	Consider $\left(\sT_{\D,\left\lfloor \frac{1}{8\lambda^2}\right\rfloor},o\right)$ the infinite rooted tree with root $o$ of degree $\D$ and all other vertices of degree $\left\lfloor \frac{1}{8\lambda^2}\right\rfloor$, and start the infection from only the root infected. Then the average total time the root is infected is bounded by
	\[
	\E\left[
	\int_0^\infty \xi_t(o) {\rm d}t
	\Big|
	\xi_0=\mathbbm{1}_o
	\right]
	\leq
	53 e^{72\lambda^2 \D}.
	\]
\end{lemma}

\begin{proof}

As in \cite{mountford2013metastable,liggett1999stochastic} and others we consider an exponentially weighted sum
\[
H_t
:=
\sum_{v \neq o} \xi_t(v) (2\lambda)^{\de(o,v)-1}.
\]
From the calculation \cite[Equation 5.1]{mountford2013metastable} we can bound its expected infinitesimal increase by
\begin{equation}\label{eq_drift}
	\frac{\E\left[ H_{t+\rm d t}- H_t | \cF_t \right]
		}{{\rm d}t}
	\leq
	\lambda \D \, \xi_t(o)
	-\frac{1}{4}
	H_t,
\end{equation}
where $\cF_t$ is the $\sigma-$algebra generated by $(H_s)_{s\le t}$.
By these dynamics, the process $H_t$ is likely to quickly go to 0 when the root is healthy, while it is likely to quickly go to the value $4\lambda \D$ (or close to this value) when the root is infected. Of course, the root also changes its state from infected to healthy at rate 1, and from healthy to infected at rate $\lambda$ times the number of infected neighbours, which is bounded above by $\lambda H_t$, as neighbours of the root have weight 1.

To bound the total time the root is infected, we first check whether its first recovery is a permanent recovery or not. If this fails, we first ask the process $H_t$ to go below level $8\lambda \D$ before checking for the next recovery and repeating the procedure.

More precisely, we introduce the sequences of stopping times $(R_k)_{k\ge 1}$, $(S_k)_{k\ge 1}$ and $(T_k)_{k\ge 0}$ defined recusrively by $T_0=0$ and for $k\ge 1$:
\begin{align*}
	R_k&:=\inf \{t\ge T_{k-1}, \xi_t(o)=0\},\\
	S_k&:=\inf \{t\ge R_k, \xi_t(o)=1\},\\
	T_k&:= \inf \{t\ge S_k, H_t\le 8\lambda \D\}.
\end{align*}
We claim that:
\begin{enumerate}
	\item Conditionally on $T_{k-1}$ being finite, the total time the center is infected on the time interval $[T_{k-1},T_k]$, which is bounded by $(R_k-T_{k-1})+(T_k-S_k)$, is in average bounded by a constant, and more precisely%
	\[
	\E\left[(R_k-T_{k-1})+(T_k-S_k)\ \big|\ T_{k-1}<+\infty\right]\le 1+16\log(1+\nicefrac 9 8).
	\]
	\item Conditionally on $T_{k-1}$ being finite, we have a probability at least $\frac{1}{4}e^{-72\lambda^2 \D}$ of having $T_k$ infinite, and more precisely of failing to reinfect the center after time $R_k$, leading to $S_k$ being infinite.
\end{enumerate}
Combining these two claims we deduce the result by
\begin{align*}
\E\left[
\int_0^\infty \xi_t(o) {\rm d}t
\right]&\le \sum_{k\ge 1} \P(T_{k-1}<+\infty) \E\left[\int_{T_{k-1}}^{T_k}\xi_t(o)\rm d t \Big| T_{k-1}<+\infty\right] \\
&\le \left(1+16\log(1+\nicefrac 9 8)\right) \cdot 4e^{72\lambda^2 \D}\le 53 e^{72\lambda^2 \D}.
\end{align*}
Thus it remains to check the claims in $(i)$ and $(ii)$.

For the first point, we clearly have $\E[R_k-T_{k-1}| T_{k-1}<+\infty]\le 1$, as the center recovers at rate 1 (if infected at time $T_{k-1}$). We continue with the following claim:
\begin{claim*}
	Write $T_{[0,8\lambda \D]}(H):=\inf\{t\ge 0,H_t\le 8\lambda \D\}$. Then, starting from any configuration with $H_0=h_0$, we have
	\[
	\E\left[T_{[0,8\lambda \D]}(H)\right]\le 16 \log\left(1+\frac {h_0}{8\lambda \D}\right).
	\]
\end{claim*}
Indeed, consider the process $Y_t=\log(1+H_t/8\lambda\D)+t/16$. For any $t< T=T_{[0,8\lambda \D]}(H)$ and so $H_t>8\lambda\D$, we can use the concavity of the log function and conditional Jensen inequality to see that the process $Y$ has nonpositive expected infinitesimal increase:
\[
\frac{\E\left[Y_{t+\rm d t}-Y_t | \cF_t\right]}
{{\rm d}t}
\leq
\frac{1}{8\lambda\D}\cdot
\frac{\lambda \D
	-\frac{1}{4}H_t}{1+ \frac{H_t}{8\lambda\D}}
+\frac{1}{16}
=
-\frac{1}{4}+\frac{1}{16}+\frac{3\lambda\D}{8\lambda\D+H_t}
\leq
0
,
\]
 and therefore $Y_{t\wedge T}$ is a nonnegative supermartingale. We deduce
\[
\E[T]\le \E[16 Y_T]\le \E[16 Y_0]=16  \log\left(1+\frac {h_0}{8\lambda \D}\right).
\]

It now follows that
\[
\begin{split}
\E[T_k-S_k | T_{k-1}<+\infty]
&\le 16  \E\left[\log\left(1+\frac {H_{S_k}}{8\lambda \D}\right)\Big| T_{k-1}<+\infty\right]\\
&\le 16  \log\left(1+\frac {\E[H_{S_k}| T_{k-1}<+\infty]}{8\lambda \D}\right),
\end{split}
\]
using again Jensen's inequality. We thus have to control the conditional expectation $\E[H_{S_k}| T_{k-1}<+\infty]$. To this end, observe that $T_{k-1}\le 8\lambda \D$ and that the expected infinitesimal increase of $H$ is bounded by $\lambda \D$ by~\eqref{eq_drift}. The recovery of the root occurs at rate 1, and so 
	\[
	\E[H_{R_k}|T_{k-1}<+\infty]\leq 9\lambda \D.
	\]
	Using again~\eqref{eq_drift} when the center is healthy, we see that the process $(H_{t\wedge S_k})_{t\ge R_k}$ is a nonnegative supermartingale, which therefore has $\E[H_{S_k}| T_{k-1}<+\infty]\le 9\lambda \D$. Note that we might have $S_k=+\infty$ if the recovery of the root is permanent, but in that case the result still holds (with $H_{S_k}=H_{\infty}=0$). Finally,
\[
\E[T_k-S_k | T_{k-1}<+\infty]\le 16 \log(1+\nicefrac{9}{8}),
\]
with the understanding $T_k-S_k=0$ if $S_k=+\infty$.
It remains to prove the second point.

	Suppose that the centre is never reinfected after time $R_k$. If this happens, by \eqref{eq_drift} we see that $(e^{t/4}H_{t})_{t \geq R_k}$ is a nonnegative supermartingale. Hence by Markov's inequality
\[
\mathbb{P}\left(
\int_{R_k}^\infty H_t {\rm d}t > 72\lambda \D
\bigg| T_{k-1}<+\infty
\right)
\leq
\frac{9\lambda \D \int_{R_k}^\infty e^{-(t-R_k)/4} {\rm d}t}{72\lambda \D}
=\frac{1}{2}
\]
and so on the complement event, recalling that we infect the root at rate bounded by $\lambda H_t$, the probability that we really don't ever see reinfection of the root is at least
$
\frac{1}{2}e^{-72\lambda^2 \D}
.
$ That is,
\[
\mathbb{P}\left(
S_k=+\infty | T_{k-1}<+\infty
\right)
\geq
\frac{1}{4}e^{-72\lambda^2 \D},
\]
which concludes the second point and thus the proof.
\end{proof}

\begin{corollary}\label{cor_infections_at_target}
	Consider $\left(\sT_{\D,\left\lfloor \frac{1}{8\lambda^2}\right\rfloor},o\right)$ the infinite rooted tree with root $o$ of degree $\D$ and all other vertices of degree $\left\lfloor \frac{1}{8\lambda^2}\right\rfloor$, with an infection sink $v$  at distance $\de(v,o)=k\geq 1$. The number of infections we expect it to receive from an infection at $o$ is 
	\[
	m(o,v) \le  27 e^{72\lambda^2 \D} (2\lambda)^k.
	\]
\end{corollary}

\begin{proof}
We let $w \sim o$ be the neighbour of the root on the path from $o$ to $v$. While infected, the root keeps infecting $w$ at rate $\lambda$, so we can bound\footnote{Note that the number of infections from $o$ to $w$ has a direct influence on the possible reinfections of $o$ from $w$ and thus on the total time the root is infected. So we do NOT condition on the total time the center is infected to obtain our upper bound on the mean number of infections from $o$ to $w$.} the mean total number of infections from $o$ to $w$ by the contact process by $\lambda \cdot 53 e^{72\lambda^2 \D}\le 27 e^{72\lambda^2 \D} (2\lambda)$. This allows to conclude the case $k=1$ where $w=v$. Suppose now $k\ge 2$, and let $(U_n)_{n\ge 1}$ be the ordered infection times from $o$ to $w$, with $U_n=+\infty$ if the number of such infections is less than $n$. As previously explained, we then have $N:=\inf \{n\ge 0, U_{n+1}=+\infty\}$ random with $\E[N]\le  27 e^{72\lambda^2 \D} (2\lambda)$. Now, any infection path from $o$ to $v$ can be decomposed into a first part of the path leading to $w$ at some time $U_n$, and a second part of the path from $w$ to $v$ without returning to the root, namely living in the subtree $\sT(w)$ obtained by disconnecting $w$ from the root. So the number of infections sent to $v$ is bounded by
\[
\sum_{n\ge 1}\1_{U_n<+\infty} X_n(w,v),
\]
where $X_n(w,v)$ is the number of infections sent to $v$ for the contact process in $\sT(w)$ started at time $U_n$ with only $w$ infected. Note that $X_n(w,v)$ is not independent from $N$, from $U_{n+1}$ or from $X_{n+1}(w,v)$, but still the $U_n$ are stopping times, and conditionally on $U_n$ being finite, we can bound the expectation of $X_n(w,v)$ by $(2\lambda)^{d(w,v)},$ as the tree $\sT(w)$ has degrees bounded by $1/8\lambda^2$. So the mean number of infections sent from $o$ to $v$ has
\begin{align*}
m(o,v)&\le \E\left[\sum_{n\ge 1}\1_{U_n<+\infty} X_n(w,v)\right] \\
&\le \E[N] (2\lambda)^{k-1}\le 27 e^{72\lambda^2 \D} (2\lambda)^k.
\end{align*}
\end{proof}

The former results holds on trees which contains vertices of degree $\left\lfloor \frac{1}{8\lambda^2}\right\rfloor$ and only one high-degree vertex. On more complicated trees with several high-degree vertices, we will bound the infection by a CP-BRW with sinks all high-degree vertices. However, in order for this CP-BRW to die out locally (and thus for the upper bound to be of any use), we will request a control on the high degrees (have them bounded by $\T$), as well as on the number of such high degree vertices. More precisely, we consider a tree where only a small proportion $q$ of the neighbours of any given vertex has degree $\T$.

\begin{definition}
	Let
	\[
	q:=\lambda^{145\rho}.
	\]
	Define $(\cB,o)$ the infinite rooted tree with root of degree $\T$ and vertices of degree either $\left\lfloor \frac{1}{8\lambda^2}\right\rfloor$ or $\T$, where for each vertex $v\in\cB$: 
	\begin{itemize}
		\item $\left\lfloor q \de(v) \right\rfloor$ of the neighbours of $v$ have degree $\T$,
		\item the remaining $\de(v)-\left\lfloor q \de(v) \right\rfloor$ have degree $\left\lfloor \frac{1}{8\lambda^2}\right\rfloor$.
	\end{itemize}
	For $v$, $w$ two vertices in $\cB$, we let $[v,w]$ be the unique path $(v_0=v,v_1,\ldots, v_{d(v,w)})$ in $\cB$ from $v$ to $w$, and $\nbig(v,w)$ be the number of big vertices on the path $[v,w]\setminus w$, so that we always have $\nbig(v,v)=0$.
\end{definition}

\begin{lemma}\label{lemma_infections_in_B}
Consider a target sink $v$ on $\cB$ at distance $\de(v,o)=k\geq 1$. 
For $\rho$ sufficiently small, the number of infections we expect $v$ to receive from $o$ has %
\[
m(o,v)\le \left(55\lambda\right)^{
\de(o,v)
-
72\rho\, \nbig(o,v)
}.
\]
\end{lemma}

\begin{proof}
We bound the infection on $\cB$ by considering the CP-BRW with sinks $\cV:=\{v \in V(\cB): \de(v)=\T\}\cup\{v\}\ni o$ and target sink $v$. 

We first consider an infection at a given sink $w\in \cV\setminus \{v\}$ and bound the number of infections it can send to other sinks. To do so, we have to simulate the contact process on the subtree $\cB_w$ of $\cB$ obtained by turning the sinks (except $w$) into leaves. The vertices in $\cB_w \setminus\{w\}$ then have degrees bounded by $\left\lfloor \frac{1}{8\lambda^2}\right\rfloor$ and we can use Corollary~\ref{cor_infections_at_target} to upper bound the mean number of infections sent to any other sink $u$ as
\[m^{\cV}(w,u)\le 27 \lambda^{-72 \rho} (2\lambda)^{d(w,u)}.\]
We will also need bounds on the number of such sinks that we can attain at any given distance from the target sink $v$.
The vertex $w$ is connected by a path of smaller degree vertices to fewer than
\[
\T \cdot \left( \frac{1}{8\lambda^2} \right)^{k-1} \cdot q
\]
other vertices $u \in\cV$ with $\de(u,v)=\de(w,v)+k$, by paths of length $k$. Also, we can reach vertices $u \in\cV$ with $\de(u,v)=\de(w,v)+k$ by paths of length $2m+k$ by initially going back up the tree: for any $m \geq 1$, we use that the shortest path on the tree between two sinks %
is directly via the most recent common ancestor to claim
\[
\left|
\left\{
u \in \cB_w \cap \cV:
\de(u,v)=\de(w,v)+k,
\de(u,w)=2m+k
\right\}
\right|
< q \left(
\frac{1}{8\lambda^2}
\right)^{m+k}.
\]

Thus 
the expected amount of infections sent $k$ levels deeper in the tree from $w$  
is bounded by
\[
 27 \lambda^{-72\rho}
 \left(
(2\lambda)^{k}
 \cdot q \T \left( \frac{1}{8\lambda^2} \right)^{k-1}
 +
 \sum_{m=1}^\infty
(2\lambda)^{2m+k}
 \cdot q \left( \frac{1}{8\lambda^2} \right)^{m+k}
 \right)
 \]
 \[
=
 27 \lambda^{-72\rho} (4\lambda)^{-k}
 \left(
 q \T \cdot 8\lambda^2
 +
 q
 \right)
\leq
 45 \lambda^{-72\rho} (4\lambda)^{-k}
 \left(
 9 \lambda^2 q \T
 \right) 
 .
\]

Similarly, the expected amount of infections sent to the level of $w$ is bounded by
\[
\sum_{m=1}^\infty
q \left(
\frac{1}{8\lambda^2}
\right)^m
\cdot
 27 \lambda^{-72\rho}
(2\lambda)^{2m}
=
 27q \lambda^{-72\rho}.
\]

For infections going closer to $v$ by $k \geq 1$ levels, there are three options:
\begin{itemize}
	\item There is some sink ancestor of $w$ (including $v$) at distance strictly less than $k$, so the infection from $w$ will have to infect this intermediate sink first, and there is simply no vertex $u$ in $\cB_w$ satisfying $d(u,v)=d(w,v)-k$.
	\item  There is a sink ancestor $u \in \cV$ with $\de(w,u)=k$, in which case we expect to infect it fewer than
	\begin{equation}\label{eq_mean_backwards}
	27 \lambda^{-72\rho}
	(2\lambda)^{k}
	\end{equation}
	times.
	\item This ancestor is still not a sink and so we consider longer paths that return to the same level.
\end{itemize}
The same calculation $\sum_{m=1}^\infty q (8\lambda^2)^{-m}
\cdot (2\lambda)^{2m}\ll 1$
however shows that the third case is bounded by the second and so we have the bound of \eqref{eq_mean_backwards} in general for the expected infection going back $k$ levels.

Recall now the notation $\left(\phi^{\cV,[v]}_n\right)_{n\in \mathbb{N}}$ from Section~\ref{sec:CP_BRW} for the BRW on the set of sinks with absorption at $v$, and define
\[
M_n
:=
\sum_{w \in \cV} \phi^{\cV,[v]}_n(w) (3\lambda)^{\de(w,v)}
\alpha^{\nbig(w,v)},
\]
for some parameter $\alpha$. We can ensure that $M_n$ is a supermartingale if the condition $\E\left[\nicefrac{M_{n+1}}{M_n}\right]\le 1$ is satisfied, and putting the above computations together gives
\[
\begin{split}
\E\left[\frac{M_{n+1}}{M_n}\right]\leq
E&=\sum_{k=1}^\infty 27 \lambda^{-72\rho}
(2\lambda)^{k} (3\lambda)^{-k} \alpha^{-1}
+
27q \lambda^{-72\rho}
+
\sum_{k=1}^\infty
 27 \lambda^{-72\rho}
(4\lambda)^{-k}
 \left(
 9 \lambda^2 q \T
 \right) 
 (3\lambda)^k \alpha\\
 &=27 \lambda^{-72\rho}\left(
 \alpha^{-1}
 \frac{2\lambda}{3\lambda-2\lambda}
 +q
 +9\lambda^2 q\T \alpha \frac{3\lambda}{4\lambda-3\lambda}
 \right)\\
 &\leq 27 \lambda^{-72\rho}\left(
 \frac{2}{\alpha}
 +q
 +27\rho q  \alpha  \log \frac{1}{\lambda} 
 \right)
\end{split}
\]
We can thus ensure $E\le 1$ by setting $\alpha=55 \lambda^{-72\rho}$, in which case the first term is less than $\nicefrac{54}{55}$ and the other two are $o(1)$ as $\lambda \downarrow 0$. The supermartingale $M_n$ then converges a.s. to $M_\infty\ge \phi^{\cV,[v]}_\infty(v)$, and as usual by optional stopping we have
\[
m^{\cV}_\infty(o,v)
\le\E_{\1_o}\left[M_\infty\right]\leq M_0
=
(3\lambda)^{\de(o,v)}
\left(55 \lambda^{-72\rho}\right)^{\nbig(o,v)}
\]
which proves the claim.
\end{proof}

Write $(\cA,o)$ for the rooted Galton-Watson tree with power law offspring distribution but truncated at $\T$.
We now have all prerequesite results to upper bound the survival probability on the random tree $\cA$. 

\begin{theorem}\label{theorem_moderate_conclusion}
For any $\delta>0$, the contact process with
\[
\xi_0=\mathbbm{1}_o
\]
on $(\cA,o)$ will survive to infinity with probability less than
\[
e^{-\lambda^{-2+\delta}}
\]
when $\rho$ and $\lambda$ are sufficiently small.
\end{theorem}

\begin{proof}

The tree $\cA$ can naturally contain several high-degree vertices, with degree inbetween $\left\lfloor \frac{1}{8\lambda^2}\right\rfloor$ and $\T$. However, if we consider one vertex of degree $d$ in $\cA$, its children will be high-degree vertices independently with probability
\[
p:=\mathbb{P}\left(
D \geq \left\lfloor \frac{1}{8\lambda^2}\right\rfloor
\right)
\sim
\cq \left( 8\lambda^2  \right)^{\tau-2},
\]
which is much smaller than $q$ when $\rho$ is small. It will then be extremely unlikely that it has more than $q\T$ high-degree neighbours (or $q \left\lfloor \frac{1}{8\lambda^2}\right\rfloor$ if $d\le \left\lfloor \frac{1}{8\lambda^2}\right\rfloor$). 
We now choose a distance $k$ as
\[k:=\left\lceil \lambda^{150\rho-2} \right\rceil.
\]
We consider the ball $B(o,k)$ in $\cA$, namely the subtree of $\cA$ consisting of all vertices at distance at most $k$ from $o$, as well as its ``boundary'' $\partial B(o,k)$, which consists of all vertices as distance exactly $k$ from $o$. In order to bound the survival probability on $\cA$ we:
\begin{itemize}
	\item show that $B(o,k)$ is likely to be a subtree of $\cB$. We write $\cE$ for this event.
	\item on the event $\cE$, bound the probability of hitting $\partial B(o,k)$, and thus the survival probability on $\cA$, thanks to Lemma~\ref{lemma_infections_in_B}.
\end{itemize}
Recall in this case that we assume $\T(\lambda)\leq \rho \frac{ \log \frac{1}{\lambda}}{\lambda^2}$ and so in the strict interior of this ball we cannot have more than
\begin{equation}\label{eq_maximal_ball}
\sum_{i=0}^{k-1}\T^i \sim \T^{k-1}
\leq
\exp\left(
3\lambda^{150\rho-2}\log \frac{1}{\lambda}
\right)
\end{equation}
vertices. To fit the ball $B(o,k)$ inside $\cB$, we simply check successively that each vertex does not have too many children of large degree. For a given vertex, we have the uniform bound for the failure probability
\[
\begin{split}
\max_{m < \left\lfloor\frac{1}{8\lambda^2}\right\rfloor}
&\mathbb{P}\left(
\operatorname{Bin}\left(
m,p
\right)
\geq
\left\lfloor q\left\lfloor\frac{1}{8\lambda^2}\right\rfloor\right\rfloor-1
\right)
\vee
\max_{\left\lfloor\frac{1}{8\lambda^2}\right\rfloor \leq n < \T}
\mathbb{P}\left(
\operatorname{Bin}\left(
n,p
\right)
\geq
\left\lfloor q\T\right\rfloor-1
\right)\\
&\leq\mathbb{P}\left(
\operatorname{Bin}\left(
\left\lfloor\frac{1}{8\lambda^2}\right\rfloor -1,p
\right)
\geq
\left\lfloor q\left\lfloor\frac{1}{8\lambda^2}\right\rfloor\right\rfloor-1
\right)
\vee
\mathbb{P}\left(
\operatorname{Bin}\left(
\T-1,p
\right)
\geq
\left\lfloor q\T\right\rfloor-1
\right)\\
&=
\mathbb{P}\left(
\operatorname{Bin}\left(
\left\lfloor\frac{1}{8\lambda^2}\right\rfloor-1,p
\right)
\geq
\left\lfloor q\left\lfloor\frac{1}{8\lambda^2}\right\rfloor\right\rfloor-1
\right)\\
&\leq 
\mathbb{P}\left(
\operatorname{Pois}\left(
\frac{p}{8\lambda^2}
\right)
\geq
\left\lfloor q\left\lfloor\frac{1}{8\lambda^2}\right\rfloor\right\rfloor-1
\right)
\leq
\frac{e^{\frac{p}{8\lambda^2}\left( e-1 \right)}}{e^{\left\lfloor q\left\lfloor\frac{1}{8\lambda^2}\right\rfloor\right\rfloor-1 }}
\leq
\exp\left(-\frac{q}{9\lambda^2} \right)
\end{split}
\]
for $\rho$ sufficiently small. Thus the probability that we ever see too many big children is bounded by the expected number of times this happens, which (recalling $q=\lambda^{145\rho}$) provides the bound
\[
\mathbb{P}(\cE^c)\le 2\T^{k-1}\exp\left(-\frac{1}{9} \lambda^{145\rho-2} \right)
\leq
\exp\left(-\frac{1}{10} \lambda^{145\rho-2}\right).
\]

With the event $\cE$ we write, as in the notation of \eqref{equation_branching_sum}, \[m_\cE(o,\partial B(o,k)):=\sum_{v \in \partial B(o,k)}m_\cE^{\partial B(o,k)}(o,v)\] for the mean infections sent, restricted to the event $\cE$. 
 We can then use Lemma~\ref{lemma_infections_in_B} to obtain that the mean number of infections hitting any $v \in \partial B(o,k)$ has
\[
m_\cE^{\partial B(o,k)}(o,v) \leq 
\left(55\lambda\right)^{
	\de(o,v)
	-
	72\rho\, \nbig(o,v)
} \leq 
\left(55\lambda\right)^{
(1
-
72\rho )k
}.
\]

We conclude by linearity of the expectation
\[
m_\cE(o,\partial B(o,k))
\leq
\left(55\lambda\right)^{
(1
-
72\rho )k
}
\E\left[
\left|\partial B(o,k)\right| \mathbbm 1_\cE
\right]
\leq 
\left(55\lambda\right)^{
	(1
	-
	72\rho )k
}
\E\left[
\left|\partial B(o,k)\right|
\right].
\]

It remains to control the expected number of targets. The $\tau \in (2.5,3)$ offspring mean is
\[
\begin{split}
\E[D] &\leq \int_0^{\T-1} \mathbb{P}(D>x) {\rm d}x
\sim \int_0^\T \left(
\cq x^{2-\tau} - \cq \T^{2-\tau}
\right) {\rm d}x\\
&=\cq\left(\frac{1}{3-\tau}-1\right)\T^{3-\tau}
\sim \cq \rho^{3-\tau} \left(\frac{\tau-2}{3-\tau}\right) \lambda^{2\tau-6} \log^{3-\tau}\lambda.
\end{split}
\]
and so
\[
\begin{split}
m_\cE(o,\partial B(o,k))
&\leq
\left(55\lambda\right)^{
(1
-
72\rho )k
}
\E\left[
D
\right]^k\\
&\leq\exp\left(
-\left(
2\tau-5-73\rho
\right)
\lambda^{150\rho-2}\log\frac{1}{\lambda}
\right)
\end{split}
\]
where we require small $\rho$ to make a useful bound. 
Altogether by the union bound and Markov's inequality
\[
\begin{split}
\mathbb{P}(\text{infection survives})
&\leq
\mathbb{P}(\cE^c)+
m_\cE(o,\partial B(o,k))\\
&\leq \exp\left(-\frac{1}{10} \lambda^{145\rho-2}\right)
+
\exp\left(
-\left(
2\tau-5-73\rho
\right)
\lambda^{150\rho-2}\log\frac{1}{\lambda}
\right)\\
&\leq
e^{-\lambda^{-2+\delta}}
\end{split}
\]
for sufficiently small $\rho$.

Finally, in the excluded $\tau=3$ case we instead calculate $\E[D] \sim 2\cq \log \tfrac{1}{\lambda}$. This has slower logarithmic growth as $\lambda\downarrow 0$ than the polynomial growth we saw above for any $\tau \in (2.5,3)$, and so by comparison to one of the other $\tau$ values the same bounds still apply.
\end{proof}

\section{\texorpdfstring{Power-law exponent $\tau>3$}{Power-law exponent t > 3} }\label{sec:large_tau}

In this section we address the final parameter regime $\tau>3$ of our main theorem. First in Section \ref{ssec:large_thresh_large} we lower bound the survival probability in the supercritical case. The upper bound for the subcritical case is addressed in Section \ref{ssec:small_thresh_large}, in part by Lemma \ref{Small_Tree_S} for the event of the environment and then at the end of the section we argue for the subevent of the process. For this, we introduce the environment offspring mean parameter
\[
\mu:=\E(X)>1
\]
which was infinite in the previous sections.

\subsection{Large truncation threshold}\label{ssec:large_thresh_large}

For $\tau>3$, an significant difference with the case $\tau\le 3$ is that big vertices are (typically) not direct neighbours but further apart. More precisely, we have to look at the neighbourhood of a big vertex up to distance $c\log(\nicefrac{1}{\lambda})$ for some $c>(2\tau-6)/\log \mu$ in order to typically find other big vertices. The larger threshold $\T=\rho \lambda^{-2} \log^{2} ( \nicefrac{1}{\lambda} )$ however allows for the local survival at a big vertex to last so long that we typically will infect those close big vertices.

Calling ``big vertex'' a vertex with degree larger than $\nicefrac{\T}{2}$, the large truncation result in the case $\tau>3$ follows from the following proposition:
\begin{proposition}
	For large enough $\rho$ and small $\lambda>0$, the survival probability given the root is big, is bounded below by $\nicefrac{1}{3}$.
\end{proposition}
We start from a big vertex and look at $D_R$ the total amount of vertices at distance exactly $R= \left\lfloor c \log (\nicefrac{1}{\lambda}) \right\rfloor.$ We have
\[\E[D_R]\ge \frac{\T}{2} \E[D]^{R-1},\]
With by assumption $\E[D]>1$. By \cite[I.A.2]{athreyaney}, its variance has
\[
\mathbb{V}{\rm ar} D_R \leq
\T \cdot
\mathbb{V}{\rm ar} D\
\frac{\E[D]^{2R-1}}{\E[D]-1}.
\]
Note that $\mathbb{V}{\rm ar} D$ is of order $O(\T^{4-\tau})=o(\T)$ for $\tau\in(3,4)$, and of logarithmic or constant order if $\tau\ge 4$, thus $\mathbb{V}{\rm ar}\left| \partial B\left(o,R\right) \right| \ll \left(\E \left| \partial B\left(o,R\right) \right| \right)^2$ and so by the second moment method we have
\[
D_R \geq \frac{\T}{3} \E[D]^R
\]
with high probability in $\lambda$. Then each of these vertices is independently big with probability $\Theta(\T^{-(\tau-2)})$. Now, we have
\[
\T^{-(\tau-3)} \E[D]^R \geq
\nicefrac{
\lambda^{(2 \tau-6)-c \log \E[D]}
}{
\left(
\log \frac{1}{\lambda}
\right)^{O(1)}
}
\gg 1,
\]
using $\E[D]\to\mu$ and $c\log\mu>2\tau-6$, and so with high probability in $\lambda$ the number of big vertices at distance exactly $R$ is of the same order.

Now we consider the infection. By \cite[Lemma 3.1(i)]{mountford2013metastable} we have probability $\nicefrac{1}{e}-o(1)$ to infect a proportion larger than ${\lambda}/{16e}$ of the neighbours of the root, and from this point we can use the following lemma to infect all nearby vertices, provided the truncation threshold was large enough. We state here a variation of this lemma directly applicable to our settings:
\begin{lemma}[ {\cite[Lemma 3.2]{mountford2013metastable}} ]\label{lemma_star_communication}
	For large $\rho$ and small $\lambda$, for any big vertex $x$ and near vertex $y$, in the sense
	\[
	\deg(x) \geq \frac{\rho \log^2 \frac{1}{\lambda}}{2 \lambda^2}
	\quad\text{ and }\quad
	\de(x,y) \leq c \log \frac{1}{\lambda},
	\]
	we have
	\[
	\mathbb{P}\left(
	\exists t>0 : \frac{\xi_t(B(y,1))}{\lambda \deg(y)}> \frac{1}{16 e}
	\Bigg|
	\frac{\xi_0(B(x,1))}{\lambda \deg(x)}> \frac{1}{16 e}
	\right)
	>1-2e^{-c \log^2 \frac 1 \lambda}
	\]
	where $\xi_t$ applied to a set of vertices denotes the total number of infected vertices in that set.
\end{lemma}

Taking $\rho$ large as in this Lemma, by the union bound, we then succeed in infecting all the nearby vertices $y$ of the Lemma with failure probability bounded by
\[
1-\left(
1-2e^{-c \log^2 \frac 1 \lambda}
\right)^{|B(x,c \log \nicefrac{1}{\lambda})|}
=
O^{\log \frac{1}{\lambda}}\left(e^{-c\log^2 \frac{1}{\lambda}}
\lambda^{(2 \tau-6)-c \log \E[D]}
\right)\to 0.
\]
Altogether, we thus have probability $\nicefrac{1}{e}-o(1)$ to infect all nearby vertices, and in particular all the big vertices at distance $R$ from the root. Now, from each of these infected big vertices, we can imagine cutting the parent edge and start an independent infection process in its own branch, thus obtaining a branching process of big vertices ever hit by the infection, which has an offspring mean tending to infinity and so is clearly supercritical, with survival probability larger than  $\nicefrac{1}{e}-o(1)>\nicefrac{1}{3}.$

\subsection{Small truncation threshold}\label{ssec:small_thresh_large}

For small $\rho$, the ``small'' truncation threshold $\T=\rho \lambda^{-2}\log^2 (\nicefrac{1}{\lambda})$ is no more sufficient to pass the infection from a big vertex to another big vertex at distance $c\log(\nicefrac{1}{\lambda})$ apart, and the previous survival strategy fails, but of course this is not sufficient to prove extinction of the epidemic.

Observe that for large truncation threshold, a typical environment ensuring a high survival probability for the contact process was to consider the root being a big vertex, and other big vertices being found at distance roughly $c\log(\nicefrac{1}{\lambda})$ apart. In this section, we will in particular want to show that for a small truncation threshold, 
such an environment provides only a small survival probability. 
However, the probability of having the root and one direct neighbour of the root both big vertices is polynomially small in $\lambda$, so in order to obtain our main result we also should upper bound the survival provability of the infection on such an environment\footnote{Similarly, the event that the root has 5 big neighbours -- or any given constant number of big neighbours -- is also polynomially unlikely, and so we also have to bound the survival probability on this event.}.

\begin{figure}[t!]
	\centering
	\centerline{\includegraphics[width=0.8\textwidth ]{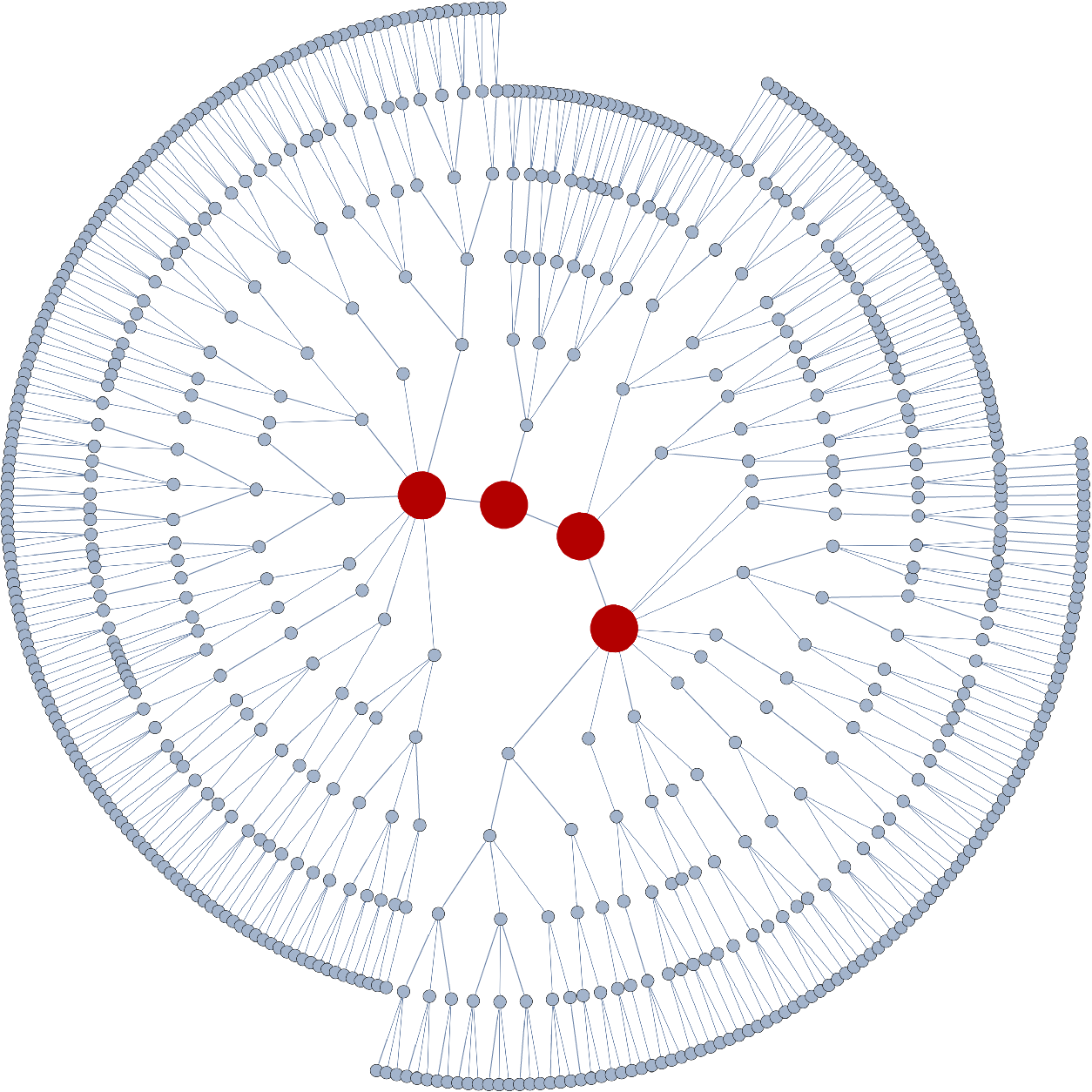}}
	\caption{We explore the $\lceil \delta \log(\nicefrac{1}{\lambda})\rceil$-ball of one (big) vertex of high degree, discovering another vertex of high degree. Thus $|\B|=2$ in this exploration but the four highlighted vertices are the set $\cC$, around each of which we explored a ball of radius $\lceil \delta \log(\nicefrac{1}{\lambda})\rceil$ and found no further big vertices. Note also  $\cC^+$ is the set of every vertex in this picture. In our proof, the whole set $\cC$ acts as a big vertex for which we want to control the number of infections sent far away as in Subsection~\ref{sec:small_truncation_threshold_in_2.5-3}. Colour available online.}
\end{figure}

Inspired by these heuristics, we consider $\delta>0$ some parameter, with value to be determined later. We consider the core $\cC$ the smallest subtree containing the root such that any vertex $x \notin \cC$ with $\de(x,\cC)<\delta \log(\nicefrac{1}{\lambda})$ is small in the sense it has degree less than $1/8\lambda^2$. We also denote by $\B$ the set containing the root and all other big vertices in $\cC$, and 
\[
\cC^+=B(\cC,\lceil \delta \log(\nicefrac{1}{\lambda})\rceil),
\]
the $\lceil \delta \log(\nicefrac{1}{\lambda})\rceil$-neighbourhood of $\cC$. %
Note this tree $\cC^+$ contains no big vertex outside of $\cC$. Note also that we have
\[|\cC|\le |\B| \lceil\delta \log(\nicefrac{1}{\lambda})\rceil,\]
as we can construct $\cC$ by starting from the root and sequentially adding each big vertex and a path of length at most $\lceil\delta \log(\nicefrac{1}{\lambda})\rceil$ linking it to the existing tree.

Note that in this construction, nothing guarantees that $|\B|$, $|\cC|$ or $|\cC^+|$ are small or even finite, however part of the proof will be to show that with an appropriate choise of $\delta$ and $\rho$ and with probability $1-o(1)$ in the limit of small $\lambda$, these sizes can be controlled, and more precisely
\[
|\B|\le \frac \delta {400 \rho},\qquad |\cC^+|\le \frac \delta {400 \rho} \T^{1+\delta \log \mu}.
\]
An environment satisfying these restrictions will be called a ``good environment''. Thus in a good environment we have fewer than a (large) constant number of large vertices, and then their average ball size is less than $\omega(1)$ factor larger than expectation. Lemma~\ref{Small_Tree_S} below bounds the probability of observing a bad environment.

 \begin{lemma} \label{Small_Tree_S}
	Choose $C>0$ arbitrarily large. If $\delta$ and then $\rho$ are chosen small enough, the environment is good with failure probability bounded by $\lambda^C$ as $\lambda\to 0$.
\end{lemma}

We postpone the proof of this lemma to the end of the section, and continue with our main proof.
Taking $C>0$ arbitrary as in our main result, we can use Lemma~\ref{Small_Tree_S} to fix $\delta$ and $\rho$ such that we have a good environment with failure probability bounded by $\lambda^C$. From now on, we work conditionally on a good environment, and aim to upper bound the survival probability of the infection in that case. We will more precisely bound the probability of ever hitting a vertex of $\cC^+$ at maximal distance $\lceil \delta \log(1/\lambda)\rceil$ from $\cC$.

Heuristically, we follow the same guidelines as in Subsection~\ref{sec:small_truncation_threshold_in_2.5-3}, with the complication that the central big vertex of that section is replaced by the whole set $\cC$.
We denote by $\sT(\cC,\B)\geq \cC$ the infinite tree containing $\cC$ in which
\[
\begin{aligned}
\forall v \in \B \quad &\de(v)=\T,\\
\forall v \in V(\sT) \setminus \B \quad &\de(v) =  \left\lfloor \frac{1}{8\lambda^2}\right\rfloor.
\end{aligned}
\]
Write also $\partial \cC$ for the set of edges of $\sT(\cC,\B)$ at the boundary of $\cC$, which has
 \[ |\partial \cC|=|\cC|\left\lfloor \frac{1}{8\lambda^2}\right\rfloor-|\cC|+1+ |\B|\left(\T- \left\lfloor \frac{1}{8\lambda^2}\right\rfloor\right)\ge \frac {|\cC|}{10\lambda^2},\]
as well as
\begin{equation}\label{eq_nice_boundary}
|\partial \cC|\le |\B| \T + (|\cC|-|\B|) \frac 1 {8 \lambda^2}\le |\B|(\T+o(\T))
\leq  \frac{\delta \T}{300\rho}
=
\frac{\delta \log^2 \frac{1}{\lambda}}{300\lambda^2}.
\end{equation}

Similarly as in the proof of Lemma~\ref{lemma_expected_infectious_period_general}, we work on the tree $\sT(\cC,\B)$ and consider the exponentially weighted sum
\[
H_t
:=
\sum_{v \notin \cC} \xi_t(v) (2\lambda)^{\de(v,\cC)-1},
\]
which controls the infection outside of $\cC$, and %
has expected infinitesimal increase bounded by
\begin{equation}\label{eq_drift2}
	\frac{\E\left[ H_{t+\rm d t}- H_t | \cF_t \right]
	}{{\rm d}t}
	\leq
	\lambda |\partial \cC| \, \1_{\xi_t(\cC)>0}
	-\frac{1}{4}
	H_t.
\end{equation}
So the process $H_t$ is likely to quickly go to 0 when $\cC$ is healthy, and is likely to quickly go below $4\lambda |\partial \cC|$ (or close to this value) when $\cC$ is infected. The set $\cC$ is also infected at rate bounded by $\lambda H_t$, however we cannot simply argue that it recovers at rate 1, and so here is a limit of the similarity with Lemma~\ref{lemma_expected_infectious_period_general}. Instead, we also have an infection process on $\cC$
\[
U_t:=\sum_{v \in \cC} \xi_t(v),
\]
with expected infinitesimal increase bounded by
\begin{equation}\label{eq_drift_for_U}
	\frac{\E\left[ U_{t+\rm d t}- U_t | \cF_t \right]
	}{{\rm d}t}
	\leq
	\lambda H_t- U_t +\lambda |\cC|U_t
\end{equation}
using of course that the number of edges of the tree $\cC$ is $|\cC|-1\le |\cC|$. 
In our proof we have to study the dynamics of both processes $H$ and $U$, with the difficuly that the dynamics of any of these two processes depends on the other. In order to simplify the picture, we would like to get rid of the $H$-dependence of the dynamics of $U$. It is possible if we have an upper bound on $H$ which can serve as an upper bound for the dynamics of $U$, and to this end we will use the following lemma, which states that $H$ is likely to stay a long time below level $20\lambda |\partial \cC|$.

\begin{lemma}\label{lemma_H_bound}
	For the process $H$ defined on the tree $\sT(\cC,\B)$ with $\B$ nonempty, from any initial configuration $\xi_0$ with $H_0<15\lambda|\partial \cC|$, and for $\lambda$ sufficiently small, we find
	\[
	\mathbb{P}_{\xi_0} \left(
	\sup_{t\leq 7^{\lambda|\partial \cC|}}H_t > 20\lambda|\partial S|
	\right)
	\leq \exp\left(-\frac{1}{100\lambda} \right).
	\]
\end{lemma}

We postpone also the proof of this lemma to the end of the section, and continue with an adaptation of Lemma~\ref{lemma_expected_infectious_period_general} and Corollary~\ref{cor_infections_at_target} to our new settings.

\begin{lemma}\label{lemma_X_infectious_period}
	Assuming $\lambda$ is sufficiently small, on the tree $\sT(\cC,\B)$ we have
	\[
	\E_{\mathbbm{1}_{\cC}}\left[
	\int_0^\infty \mathbbm{1}_{\xi_t(\cC)>0} \, {\rm d}t
	\right]
	\leq
	3 e^{200\lambda^2|\partial \cC|}.
	\]
\end{lemma}

\begin{proof}
	We introduce the sequences of stopping times $(\tau_k)_{k\ge 1}$, $(R_k)_{k\ge 1}$, $(S_k)_{k\ge 1}$ and $(T_k)_{k\ge 0}$ defined recursively by $T_0=0$ and for $k\ge 1:$
	\begin{align*}
		\tau_k&:=\inf\{t\ge T_{k-1}, H_t\ge 20\lambda|\partial \cC|\},\\
		R_k&:=\inf\{t\ge T_{k-1}, \xi_t(\cC)=0\} \wedge \tau_k,\\
		S_k&:= 	\begin{cases}
				R_k & \text{if } R_k=\tau_k\\
				\inf\{t\ge  R_k, \xi_t(\cC)>0\} & \text{if }R_k<\tau_k,
				\end{cases}\\
		T_k&:=\inf\{t\ge S_k, H_t\le 15 \lambda |\partial \cC|\}.
	\end{align*}
	In words, starting from a time $t=T_{k-1}$ with $H_t\le 15 \lambda |\partial \cC|$, we wait until either $H$ goes above $20 \lambda |\partial \cC|$ or $\cC$ recovers. If we first observe $H$ going above the threshold, we just wait until it goes down below $15\lambda |\partial \cC|$ again (at time $T_k$). If on the contrary we first have recovery of $\cC$, we first check whether it is a permanent recovery (in which case we have $S_k=T_k=+\infty$) and if not, after the reinfection of $\cC$ we still wait until $H$ goes down below $15\lambda |\partial \cC|$ at time $T_k$ to repeat the procedure. We claim that:
	\begin{enumerate}
		\item Conditionally on $T_{k-1}$ being finite, we have the upper bound
		\[
		\E\left[(R_k-T_{k-1})+(T_k-S_k) \Big| T_{k-1}<\infty\right]\le 15+e^{40\lambda^2 |\partial \cC|}
		\]
		for the average time $\cC$ is infected on $[T_{k-1},T_k]$.
		\item Conditionally on $T_{k-1}$ being finite, we have a probability at least $\frac 2 5 e^{-160\lambda^2 |\partial \cC|}$   of having $T_k$ infinite, actually by observing a permanent recovery of $\cC$, leading to $S_k$ being infinite.
	\end{enumerate}
	Combining the two claims provides the upper bound
	\[
	\E\left[
	\int_0^\infty \mathbbm{1}_{\xi_t(\cC)>0} {\rm d}t
	\Bigg|
	\xi_0=\mathbbm{1}_{\cC}
	\right]
	\leq
	\left(15+e^{40\lambda^2|\partial \cC|}\right)
	\left(
	\frac 5 2 e^{160\lambda^2|\partial \cC|}
	\right),
	\]
	and thus the result for small $\lambda$.
	
	To check the claim in $(i)$, we repeat the supermartingale argument in Lemma~\ref{lemma_expected_infectious_period_general}, but now with the process $11\log(1+H_t/15\lambda |\partial \cC|)+\nicefrac{t}{11}$, which has nonnegative infinitesimal increase when $H$ is above $15\lambda |\partial \cC|$, to deduce
	\[
	\E\left[T_k-S_k  \Big| T_{k-1}<\infty\right]\le 11\log\left(1+\frac {\E[H_{S_k}|T_{k-1}<\infty]}{15\lambda|\partial \cC|}\right).
	\]
	Still working implicitly conditionally on $\{T_{k-1}<+\infty\}$, we bound the inner expectation by
	\[
	\begin{array}{rlclcl}
	\E[H_{S_k}]&\le& \E[H_{S_k}\1_{R_k=\tau_k}]&+&\E[H_{S_k}\1_{R_k<\tau_k}]&\\
	&\le &(1+20\lambda|\partial \cC|)&+&20\lambda |\partial \cC|&\le 41\lambda |\partial \cC|,
	\end{array}
	\]
	where in the second line we used that the jumps of $H$ are bounded by one to upper bound the process by $1+20\lambda|\partial \cC|$ at time $\tau_k$, and in the second term the supermartingale property of the process $H$ when $\cC$ is healthy which gives a nonpositive expected increment on the time interval $[R_k,S_k]$. Thus
	\[
	\E\left[T_k-S_k  \Big| T_{k-1}<\infty\right]\le 11\log(1+\nicefrac {41}{15})\le 15.
	\]
	In order to bound $\E[R_k-T_{k-1}|T_{k-1}<+\infty]$, it suffices to bound the recovery time of $\cC$, or equivalently the hitting time of 0 for the process $U$, under the hypothesis that $H$ stays bounded by $20\lambda|\partial \cC|$.  Recalling~\eqref{eq_drift_for_U}, we can then bound its infinitesimal increase by $20\lambda^2 |\partial \cC|-U_t/2$, using that for small $\lambda$ we have  $2\lambda |\cC|\le \frac {2 \delta}{400\rho}\lambda \lceil \delta \log(1/\lambda)\rceil \le 1$. Further though, this upper bound is just a birth-death chain on $\{0\}\cup[|\cC|]$ which by detailed balance has a stationary distribution $\pi$ with
	\[
	\forall k \in [|\cC|]
	\quad
	40\lambda^2|\partial \cC|
	\pi(k-1)
	=
	k\pi(k)
	\implies
	\frac{\pi(k)}{\pi(0)}
	=
	\frac{\left( 40\lambda^2|\partial \cC| \right)^k}{k!}
	\]
	so the edge $\{k-1,k\}$ has conductance (or ergodic flow rate, see \cite[Section 3.3]{aldous-fill-2014} for details)
	\[
	\pi(0)\frac{\left( 40\lambda^2|\partial \cC| \right)^k}{2(k-1)!}
	\]
	and we bound the hitting time of $0$ by the resistance (inverse conductance) of the path
	\[
	\begin{aligned}
	\frac{1}{\pi(0)}\sum_{k=1}^{|\cC|}\frac{2(k-1)!}{\left( 40\lambda^2|\partial \cC| \right)^k}
	&=
	\sum_{k=1}^{|\cC|}\frac{2(k-1)!}{\left( 40\lambda^2|\partial \cC| \right)^k}
	\sum_{j=1}^{|\cC|}\frac{\left( 40\lambda^2|\partial \cC| \right)^j}{j!}\\
	&\leq
	|\cC|
	\frac{2(|\cC|-1)!}{
		\left( 40\lambda^2|\partial \cC| \right)^{|\cC|}
	}
	\sum_{j=1}^{\infty}\frac{
		\left( 40\lambda^2|\partial \cC| \right)^j}{j!}
	\leq
	\frac{2(|\cC|)! e^{40\lambda^2|\partial \cC|}}{
		\left( 40\lambda^2|\partial \cC| \right)^{|\cC|}
	}.
	\end{aligned}
	\]
	
	Then, use that $|\partial \cC| \geq \frac{|\cC|}{10\lambda^2}$ and $n! \leq \frac{n^{n+1}}{e^{n-1}}$ to bound this hitting time by
	\[
	\frac{2(|\cC|)! e^{40\lambda^2|\partial \cC|}}{
		\left( 40\lambda^2|\partial \cC| \right)^{|\cC|}
	}
	\leq
	\frac{2(|\cC|)! e^{40\lambda^2|\partial \cC|}}{
		\left( 4|\cC| \right)^{|\cC|}
	}
	\leq
	2|\cC|e
	{(4e)}^{-|\cC|}
	e^{40\lambda^2|\partial \cC|}
	\leq
	e^{40\lambda^2|\partial \cC|},
	\]
	and we have shown the claim in $(i)$.
	
	For the second point, we first note that we have $R_k<\tau_k$ with probability at least $4/5$. %
	Indeed, by Lemma~\ref{lemma_H_bound} and Markov's inequality, we will have hitting of $U=0$ before time $7^{\lambda |\partial \cC|}$, and hitting of $H>20\lambda|\partial \cC|$ after time  $7^{\lambda |\partial \cC|}$, with probability at least
	\[
	1-\exp\left(-\frac{1}{100\lambda} \right)-\exp\left( \left( 40 \lambda - \log 7 \right)\lambda |\partial \cC| \right)
	\geq 1-\exp\left(-\frac{1}{5\lambda} \right)\ge \frac 4 5.
	\]
	
	Now, conditionally on $R_k$ being finite and smaller than $\tau_k$, we can condition on $\xi_{R_k}$ (a configuration which has in particular $H_{R_k}\le 20\lambda |\partial \cC|$), and observe that $(H_{t\wedge S_k} e^{(t\wedge S_k)/4})_{t\geq R_k}$ is a supermartingale to deduce
	\[
	\mathbb{P}\left(
	\int_{R_k}^{S_k} H_t>160\lambda|\partial \cC|\, \Big|\, \xi_{R_k}
	\right)
	\leq
	\frac{\E\left[
		\int_{R_k}^{S_k} H_t {\rm d}t \, \Big|\, \xi_{R_k}
		\right]}{160\lambda|\partial \cC|}
	\leq
	\frac{20\lambda|\partial \cC|
		\int_{0}^\infty e^{-t/4} {\rm d}t
	}{160\lambda|\partial \cC|}
	=\frac 12 .
	\]
	If this limit is not exceeded, the total rate of reinfection of $\cC$ is then bounded by $\lambda \cdot 160\lambda|\partial \cC|$, and so we have found that the recovery of $\cC$ is permanent with probability at least $\frac{1}{2}e^{-160\lambda^2|\partial \cC|}
	$, which allows to conclude the claim in $(ii)$, and thus our proof.
\end{proof}	

\begin{corollary}\label{corollary_escaping_X}
	On the tree $\sT(\cC,\B)$, generalise our CP-BRW notation \eqref{equation_branching_sum} so that $m(\cC,v)$ denotes the mean absorption at $v$ with the good environment $\cC$ initially fully infected. When the target sink $v$ is at distance $\lceil \delta \log(1/\lambda)\rceil$ from $\cC$, we find
	\[
	m(\cC,v) \le e^{-\frac{1}{4}\delta\log^2 \frac{1}{\lambda}}.
	\]
\end{corollary}

\begin{proof}
	Given the larger distance which is here necessary to prevent local divergence of the CP-BRW, we just repeat the proof of Corollary~\ref{cor_infections_at_target} to see that the expected number of infections absorbed at $v$ has
	\begin{align*}
	m(\cC,v)  \le 3 e^{200\lambda^2 |\partial \cC|} (2\lambda)^{\lceil \delta \log(1/\lambda)\rceil-1}.
	\end{align*}
	Now recalling from \eqref{eq_nice_boundary} the bound $|\partial \cC|\le \frac{\delta \log^2 \frac{1}{\lambda}}{300\lambda^2}$ (because we assumed for this corollary that the environment is good), we have:
	\[
	3 e^{200\lambda^2|\partial \cC|}(2\lambda)^{\lceil \delta \log(1/\lambda)\rceil-1}
	\leq
	3 e^{- \frac{1}{3}\delta\log^2 \frac{1}{\lambda}+ (1+\delta \log 2) \log \frac{1}{\lambda}} \le e^{-\frac{1}{4}\delta\log^2 \frac{1}{\lambda}}.
	\]
\end{proof}

By Corollary~\ref{corollary_escaping_X}, on a good environment the probability that the infection ever reaches any given vertex $v$ of $\cC^+$ at maximal distance $\lceil \delta \log(1/\lambda)\rceil$ from $\cC$ is bounded by $e^{-\frac{1}{4}\delta\log^2 \frac{1}{\lambda}}$. Recalling that this good environment has $|\cC^+|\le \frac{\delta}{400\rho} \T^{1+\delta \log \mu}$, by the union bound the survival probability of the infection on this good environment is bounded by
\[
\left\lfloor\frac{\delta}{400\rho} \T^{1+\delta \log \mu}\right\rfloor e^{-\frac{1}{4}\delta\log^2 \frac{1}{\lambda}}
\]
This quantity is smaller than any polynomial in $\lambda$ as $\lambda\to 0$, thus yielding the result.

\begin{proof}[Proof of Lemma~\ref{Small_Tree_S}]
	First, by the monotonicity of the Galton-Watson tree it suffices to show the result on the tree with maximal root degree $\de(o)=\T-1$. We then explore the ball around this root of radius $\left\lceil \delta \log \nicefrac{1}{\lambda} \right\rceil$ for other big vertices (with degree at least $\frac{1}{8\lambda^2}$), as well as the balls around paths between big vertices.
	
	Continue iteratively to explore the ball of radius $\left\lceil \delta \log \nicefrac{1}{\lambda} \right\rceil$ around every big vertex until this construction terminates and we have explored the whole $\cC$, revealing the big vertices and paths between them of length at most $\left\lceil \delta \log \nicefrac{1}{\lambda} \right\rceil$.
	
	To obtain the desired bounds on $|\B|$ and $|\cC^+|$, we combine two ingredients:
	\begin{itemize}
		\item We first control the size of the ball of radius $\left\lceil \delta \log \nicefrac{1}{\lambda} \right\rceil$ around a given big vertex. 
Using \cite[Theorem 1.1]{strayappendix} with 
		\[n=\left\lceil \delta \log \frac{1}{\lambda} \right\rceil
		\qquad
		\text{and}
		\qquad		 
		x=\lambda^{-2- \frac{1}{2} \delta \log \mu},
		\]
		we obtain that the probability that the boundary of a ball with root degree $\lfloor 1/8\lambda^2 \rfloor$ contains more than
		\[\T^{1+\delta \log \mu + \frac{1}{2}\delta \log \mu }\]
		vertices decays like a \emph{stretched exponential} in $\lambda$. A big vertex and the path which discovered it produce together at most \[
		\left\lceil \delta \log \frac{1}{\lambda} \right\rceil + \frac{\T}{\lfloor 1/8\lambda^2 \rfloor}
		\leq 9\rho \log^2\left( \nicefrac{1}{\lambda} \right)
		\] such balls and hence by the union bound, a total contribution to $\cC^+$ smaller than
		\[ 9\rho \log^2\left( \nicefrac{1}{\lambda} \right)
		\T^{1+\frac{3}{4}\delta \log \mu }
		\leq 
\T^{1+\frac{4}{5}\delta \log \mu }	
		\]
		with a probability that's also stretched exponential in $\lambda$, and so in particular a probability that is smaller than any polynomial in $\lambda$.
		
		By the union bound again, we get that with a probability which is still smaller than any polynomial in $\lambda$, the union of the balls of radius  $\left\lceil \delta \log \frac{1}{\lambda} \right\rceil$ around the first $\left\lfloor \frac \delta {400 \rho} \right\rfloor$ explored big vertices and paths between them contain less than $\left\lfloor \frac \delta {400 \rho} \T^{1+\delta \log \mu} \right\rfloor $ vertices.
		\item It remains to bound  $|\B|$ by $\frac \delta {400 \rho}$. To this end, it suffices to check that above $\left\lfloor \frac \delta {400 \rho} \T^{1+\delta \log \mu} \right\rfloor$ explored vertices, it is likely to find less than $\frac \delta {400 \rho}$ big vertices. Note that for small enough $\lambda$ each explored vertex is big with probability
		\[
		\mathbb{P}\left(D \geq \frac{1}{8\lambda^2}\right) < \mathbb{P}\left(X \geq \frac{1}{8\lambda^2}\right) < \cq \lambda^{2\tau-4}
		\]
		and hence in $\left\lfloor \frac \delta {400 \rho} \T^{1+\delta \log \mu} \right\rfloor$ explored vertices we find fewer than $\frac{\delta}{400\rho}$ large vertices with probability, using the usual Chernoff bound, bounded by
		\[
		\mathbb{P}\left(\operatorname{Bin}\left( \left\lfloor \frac \delta {400 \rho} \T^{1+\delta \log \mu} \right\rfloor, \cq \lambda^{2\tau-4}\right)\geq \frac{\delta}{400\rho} \right)
		\]
		\[\begin{split}
			&\leq
			\mathbb{P}\left(\operatorname{Pois}\left(  \frac{\delta}{400\rho}\rho^{1+\delta \log \mu} \cq \lambda^{2\tau-6-2\delta \log \mu}\log^{2+2\delta \log \mu} \frac{1}{\lambda}\right)\geq \frac{\delta}{400\rho}\right)\\
			&\leq
			\lambda^{ \frac{\varphi \delta}{400\rho}}
			\exp\left(
			\frac{\delta}{400\rho}\rho^{1+\delta \log \mu} \cq \lambda^{2\tau-6-2\delta \log \mu-\varphi}\log^{2+2\delta \log \mu} \frac{1}{\lambda}
			\right)
			\sim \lambda^{ \frac{\varphi\delta}{400\rho}}\leq \lambda^C
		\end{split}
		\]
		where we have set parameter
		$
		\varphi ={400 \rho C}/{\delta}
		$
		so that by taking $\delta$ small and then $\rho$ small enough we guarantee $2\tau-6-2\delta \log \mu-\varphi>0$.
	\end{itemize}
\end{proof}

\begin{proof}[Proof of Lemma~\ref{lemma_H_bound}]
	$
	M_t:=
	\left( {3}/{2} \right)^{H_t}
	$ 
	will be a supermartingale while $H_t$ is sufficiently large. Indeed, write its expected infinitesimal increase as
	\begin{equation}\label{eq_trivial_generators}
		\begin{split}
			\frac {\E\left[M_{t+\rm d t}-M_t \Big| \cF_t\right]}{\rm d t}
			&=
			\lambda |\partial \cC|\left(
			\left( \frac{3}{2} \right)^{H_t+1}-\left( \frac{3}{2} \right)^{H_t}
			\right)
			+
			\sum_{v \in \xi_t}\cL(v)
			\\
			&=
			\frac{\lambda |\partial \cC|}{2} \left( \frac{3}{2} \right)^{H_t} 
			+
			\sum_{v \in \xi_t}\cL(v), 
		\end{split}
	\end{equation}
	where again $\cF_t$ is the $\sigma$-algebra generated by $(H_s)_{s\leq t}$, and $\cL(v)$ is the contribution to the expected infinitesimal increase due to infections and recoveries by $v$. 
	
	First we consider infected vertices $v$ at distance $\de(v,\cC)\geq 3$ and so with weight $w=(2\lambda)^{\de(v,\cC)-1}$, 
	\[
	\begin{split}
		\cL(v)&\leq
		\left( \frac{3}{2} \right)^{H_t-w}\left(
		\frac{\lambda}{8\lambda^2}\left(\left( \frac{3}{2} \right)^{w (1+2\lambda)}-\left( \frac{3}{2} \right)^{w }\right)
		+
		\lambda\left(\left( \frac{3}{2} \right)^{w (1+\frac{1}{2\lambda})}-\left( \frac{3}{2} \right)^{w }\right)
		+1-\left( \frac{3}{2} \right)^{w }
		\right)\\
		&\sim \left( \frac{3}{2} \right)^{ H_t}
		\left(
		\frac{1}{8\lambda} \cdot  2\lambda w \log \frac{3}{2}
		+ \lambda \cdot \frac{w}{2\lambda} \log \frac{3}{2}
		- w \log \frac{3}{2} \right)
		= w \left( \frac{3}{2} \right)^{ H_t}
		\left(
		- \frac{1}{4} \log \frac{3}{2}
		\right)
	\end{split}
	\]
	where we were able to apply $e^x \sim 1+x$ to simplify every term in the limit $\lambda \downarrow 0$. 
	We have a different calculation when $\de(v,\cC)=2$, in which case
	\[
	\begin{split}
		\cL(v)
		&\leq (1+o(1)) \left( \frac{3}{2} \right)^{ H_t} \left(
		\frac{1}{4}  w \log \frac{3}{2}
		+ \lambda \left( \left( \frac{3}{2} \right)^{ \frac{w}{2\lambda} }-1 \right)
		- w \log \frac{3}{2}
		\right)\\
		&= w \left( \frac{3}{2} \right)^{ H_t}
		\left(-\frac{3}{4}   \log \frac{3}{2}
		+ \frac{1}{4} \right),
	\end{split}
	\]
	and in the final case $\de(v,\cC)=1$ infections of parent in $\cC$ do not count for $H_t$ and so 
	\[
	\begin{split}
		\cL(v)
		\leq
		\frac{1}{8\lambda}\left(
		\left( \frac{3}{2} \right)^{H_t + 2\lambda}
		-\left( \frac{3}{2} \right)^{ H_t}
		\right)
		+\left( \frac{3}{2} \right)^{H_t-1}-\left( \frac{3}{2} \right)^{H_t}
		\sim  w  \left( \frac{3}{2} \right)^{H_t}
		\left(
		\frac{1}{4}\log \frac{3}{2} - \frac{1}{3}
		\right).
	\end{split}
	\]
	
	By construction all of these coefficients are negative
	\[
	\frac{1}{4}\log \frac{3}{2} - \frac{1}{3}
	<
	- \frac{1}{4} \log \frac{3}{2}
	<
	-\frac{3}{4}   \log \frac{3}{2}
	+ \frac{1}{4}
	<-0.054
	<-\frac{1}{20}.
	\]
	Recalling \eqref{eq_trivial_generators}, we have that the expected infinitesimal increase of $M$ is bounded by
	\[
	\left( \frac{3}{2} \right)^{ H_t}
	\left(
	\frac{\lambda |\partial \cC|}{2}
	- \frac{1}{20} H_t 
	\right)
	\]
	and thus negative as long as $H_t> 10\lambda |\partial \cC|$. Define
	\begin{align*}
		\tau_0&:=\inf \left\{ t\ge 0 : H_t \geq 15\lambda|\partial \cC| \right\},\\
		\tau_D&:=\inf \left\{ t>\tau_0 : H_t < 10\lambda|\partial \cC| \right\},\\
		\tau_U&:=\inf \left\{ t>\tau_0 : H_t > 20\lambda|\partial \cC| \right\}.
	\end{align*}
	On $[\tau_0,\tau_D]$, the expected infinitesimal increase of $M$ is negative, and we have $H_{\tau_0}\le 15 \lambda |\partial \cC| + 1$ as well as $H_{\tau_D}\ge 10\lambda|\partial \cC|-1$, so we can apply optional stopping to deduce
	\begin{equation}\label{eq_geometric_probability}
		\gamma :=
		\mathbb{P}\left(\tau_U<\tau_D\right)
		\leq
		\frac{\left( \nicefrac{3}{2} \right)^{15\lambda|\partial \cC|+1}-\left( \nicefrac{3}{2} \right)^{10\lambda|\partial \cC|-1}}{\left( \nicefrac{3}{2} \right)^{20\lambda|\partial \cC|}-\left( \nicefrac{3}{2} \right)^{10\lambda|\partial \cC|}}
		\sim
		\left( \frac{3}{2} \right)^{-5\lambda|\partial \cC|+1}.
	\end{equation}

	We will thus typically observe many downward exits from above $15\lambda|\partial \cC|$ to below $10\lambda|\partial \cC|$ before $H_t$ can exceed $20\lambda|\partial \cC|$, and it remains to lower bound the time which each downward exit takes.
	From a particular infection state, the contact process dominates the contact process with no infection allowed and only recovery.	
	Hence, introducing a sequence $(E_i)_i$ of i.i.d. exponentials of mean $1$, we have the following stochastic lower bound
	\[
	H_{\tau_0+t}
	\succeq
	Z_t:=
	\sum_{v \in \xi_{\tau_0}}
	w(v)
	\mathbbm{1}_{E_i>t}.
	\]
	Setting $t=\log \frac{3}{2}$ we have $\mathbb{P}(E_i>t)=\frac{2}{3}$
	and so we can couple three copies of the exponential $E_i, \tilde{E}_i, \tilde{\tilde{E}}_i$ such that almost surely
	\[
	\mathbbm{1}_{E_i>t}+
	\mathbbm{1}_{\tilde{E}_i>t}+
	\mathbbm{1}_{\tilde{\tilde{E}}_i>t}=2.
	\]
	
	In this way, we can construct identically distributed and coupled $Z_t, \tilde{Z}_t, \tilde{\tilde{Z}}_t$ having almost surely
	\[
	Z_t+\tilde{Z}_t+\tilde{\tilde{Z}}_t
	=2\sum_{v \in \xi_{\tau_0}}w(v)
	\geq 30\lambda|\partial \cC|.
	\]
	In particular, the sum of probabilities $\mathbb{P}(Z_t \geq 10\lambda|\partial \cC|)+\mathbb{P}(\tilde Z_t \geq 10\lambda|\partial \cC|)+\mathbb{P}(\tilde {\tilde{Z}}_t \geq 10\lambda|\partial \cC|)$ exceeds one, but as these probabilities are equal, we deduce $\mathbb{P}(Z_t \geq 10\lambda|\partial \cC|)\geq 1/3$.
	
	All in all, we have a number of downward exits before $H_t$ can exceed $20 \lambda |\partial \cC|$ which is at least a Geometric of mean $-1+1/\gamma\gtrsim \left( {3}/{2} \right)^{5\lambda|\partial \cC|-1}$. First note, by Markov's inequality,
	\[
	\mathbb{P}\left(
	\operatorname{Geom}\left( \gamma
	\right)\leq k
	\right)
	\leq
	\mathbb{P}\left(
	\operatorname{Bin}\left(k+1, \gamma
	\right)\geq 1
	\right)
	\leq
	\gamma (k+1)
	.
	\]
	
	Of these downward exits at least binomially $1/3$ of them take time at least $\log \frac{3}{2}$, and it follows from the standard Chernoff bound $\mathbb{P}\left(
	\operatorname{Bin}\left(
	n,p
	\right)\leq np/2\right) \leq e^{-np/8}$
	that
	\[
	\begin{split}
		\mathbb{P}\left(
		\tau_U-\tau_0<7^{\lambda|\partial \cC|}
		\right)
		&\leq
		\mathbb{P}\left(
		\operatorname{Geom}\left( \gamma
		\right)\leq \frac{7^{1+\lambda|\partial \cC|}}{\log \frac{3}{2}}
		\right)
		+
		\mathbb{P}\left(
		\operatorname{Bin}\left(\left\lfloor
		\frac{7^{1+\lambda|\partial \cC|}}{\log \frac{3}{2}}\right\rfloor,
		\frac{1}{3}
		\right)
		\leq
		\frac{7^{\lambda|\partial \cC|}}{\log \frac{3}{2}}
		\right)\\
		&\leq
		\gamma \frac{7^{\lambda|\partial \cC|}}{\frac{2}{3}\log \frac{3}{2}}
		+\gamma
		+
		\exp\left(
		-\frac{7^{1+\lambda|\partial \cC|}}{24\log \frac{3}{2}}
		\right)\\
		&\leq \exp\left(-\frac{1}{100\lambda} \right)
	\end{split}
	\]
	where the final line follows from recalling the $\gamma$ asymptotic in \eqref{eq_geometric_probability}, that $\lambda |\partial \cC|\geq \frac{1}{8\lambda^2}$  and that $\frac{1}{8}\log \left( 7 \cdot (3/2)^{-5} \right) <-\frac{1}{100}$.
\end{proof}

\vspace{2em}

{\bf Acknowledgements.}
JF was supported by the \emph{Unit\'e de math\'ematiques pures et appliqu\'es} of ENS Lyon, and then by NRDI grant KKP 137490.

\printbibliography

\begin{appendices}
\noappendicestocpagenum

\section{A lemma on the contact process on star graphs}\label{appendix}

In this section we revisit slightly the proof of \cite[Lemma~3.4]{nguyen2022subcritical} to obtain the following refinement on the probability that the contact process, starting from only the center of a star graph infected, survives for an exponentially long time in $\lambda^2$ times the number of vertices in the star graph.

\begin{lemma}\label{lem:star_survival_center}
	Consider the contact process $(\xi_t)$ with infection rate $\lambda<1$ on the star graph $S_k$. Assume that $\rho$ is infected at time 0, and write $T=\int_0^{\infty} \1_{\rho \in \xi_t} \de t$ the total time the center is infected before the extinction of the epidemic on $S_k$. Then 
	\[
	\p\left(T\le e^{\eps \lambda^2 k}\right)\leq \frac{c}{\lambda^2 k},
	\]  for $\eps>0$ and $c>0$ universal constants (i.e independent of $\lambda<1$ and $k$).
\end{lemma}

\begin{remark}
	This result states that the probability that the contact process fails to survive for an exponentially long time (at the center of the star) is bounded below by $c/\lambda^2 k$. This is to be compared with \cite[Lemma~3.4]{nguyen2022subcritical}, with mainly two differences:
	\begin{enumerate}
		\item Their bound is $c\log(\lambda^2k)/\lambda^2 k$ rather than $c/\lambda^2 k$. We obtain this improvement by a slight refinement of one argument in their proof. Note that the failure probability really is of order $1/\lambda^2 k$ (when $\lambda^2 k$ large), as it is easily seen that the probability that the center will never be reinfected by one of its neighbours is already at least
		\[\frac 1 2 \left(1-e^{-\frac 1 {2\lambda^2 k}}\right).\]
		To prove this, just argue that conditionally on the first recovery time $R$ of the root, each leaf has independently probability $(1-e^{-\lambda R})\left(\frac \lambda {1+\lambda}\right)\le  \lambda^2 R$ to be infected by the root before time $R$ and then send an infection back before its own recovery. Thus, conditionally on $R \leq {1}/{2\lambda^2 k}$, the probability that the center is reinfected by neighbour is at most $\nicefrac{1}{2}$. We conclude as $R$ is a standard exponential random variable.
		\item They consider the extinction time on $S_k$ rather than here $T$ the total time the center is infected. This is a very minor difference, as their proof clearly implies that the center if infected at least half of the time interval where they ensure survival of the process, so they could as well have stated the result for $T$. In Section 5.1 however, we require the result with $T$ rather than the extinction time on the star.
	\end{enumerate} 
\end{remark}

Consider $(N_t)$ the number of leaves infected at time $t$, and $R$ the first recovery time of the root. We start with the following observation:
\begin{lemma}\label{lem:tail_N_R}
		There exists $a\ge 1$ such that for all $\lambda, k$, we have
		\[\P(N_R < \lfloor \lambda k x\rfloor)\le ax \qquad \forall x\le \nicefrac{1}{a}.\]
\end{lemma}
\begin{proof}
	We have to prove that for any positive integer $i$ smaller than $\lambda k x_0$, we have \[\P(N_R<i)\le \frac {a i} {1+\lambda k}.\] Call $M$ the total number of leaves that receive an infection on the time interval $[0,R]$, but are not necessarily infected anymore at time $R$, so $M\ge N_R$. We first treat separately the cases $R\ge 1$ and $\{R<1, M\ge \nicefrac{k}{2}\}$, on which $N_R$ is very unlikely to be small.
	
	On the event $\{R\ge 1\},$ we simply look at infections of the root on the time interval $I=[R-1,R]$ and argue that each leaf will independently receive an infection on $I$ and experience no recovery on $I$ with probability $(1-e^{-\lambda})e^{-1}\ge \nicefrac{\lambda}{3}$. We conclude with a standard Chernoff bound for Bernoulli random variables
	\[
	\P(R\ge 1, N_R< i)
	\le \P\left(\operatorname{Bin}\left(k,\nicefrac{\lambda}{3}\right)	\le \frac {\lambda k}4\right)
	\le \exp\left(-\frac {\lambda k}{32}\right)
	\le \frac a {1+\lambda k} 
	\le \frac {ai} {1+\lambda k}
	\]
	for $1\le i\le \lambda k/4$ and $a$ large.
	
	On the event $\{R<1, M\ge \nicefrac{k}{2}\}$, each of the $M$ infected leaves has probability $e^{-1}$ to experience no recovery on the whole time interval $[0,1]$ and thus be infected at time $R$. Thus we can lower bound $N_R$ by a $\operatorname{Bin}(\nicefrac{k}{2},e^{-1})$ and conclude similarly.
	
	Finally, while the center has not recovered and at least $\nicefrac{k}{2}$ have still not been infected, we infect a new leaf that then stays infected during one unit of time with rate at least $(\lambda k/2) e^{-1}\ge \lambda k/6$. Hence, the probability of the center recovering before $i$ such infection events is bounded by 
	\[i \cdot \frac 1 {1+\lambda k/6}\le \frac {6i}{\lambda k}.\]
	We deduce
	\[\P(N_R\le i, R<1,M<k/2)\le \frac {6i}{\lambda k}\]
\end{proof}

Lemma~\ref{lem:tail_N_R} is only a small improvement over \cite[Lemma~3.1]{nguyen2022subcritical}, where they have a second term in the upper bound that is decreasing in $x$ and becomes relevant only when $x$ is of order $\log(\lambda k)/\lambda k$ or smaller. We continue with essentially a rephrasing of \cite[Lemmas~3.2 and 3.3]{nguyen2022subcritical}:
\begin{lemma}  \cite[Lemmas~3.2 and 3.3]{nguyen2022subcritical}
	\label{lem:metastable_step}
	Starting from a configuration with $M$ infected neighbours at time $t$, the central vertex will be infected for at least half of the time-interval $[t,t+1]$ and there will be at least $\lambda k/100$ infected neighbours at time $t+1$ with probability at least 
	\[
	1-3 \exp\left(-\frac{\lambda M}{128}\right).
	\]
\end{lemma}
The improvement now is that instead of asking $N_R$ not to be too small so as to have a small bound in the probability of $\{N_{R+1}\le \frac {\lambda k}{100} \}$ with Lemma~\ref{lem:metastable_step}, we 
use %
Lemma~\ref{lem:metastable_step} to compute
\[
\P\left(N_{R+1}\le \frac {\lambda k}{100} \right)\le \E\left[\P\left(N_{R+1}\le \frac {\lambda k}{100} \bigg| N_R\right)\right]
\le \E\left[3 \exp\left(-\frac{\lambda N_R}{128}\right)\right],
\]
where by Lemma~\ref{lem:tail_N_R} we have that $(1+N_R)/\lambda k$ is stochastically bounded below by $U$ a uniform random variable on $[0,\nicefrac{1}{a}]$, and thus
\[
\P\left(N_{R+1}\le \frac {\lambda k}{100}\right)\le 3 e^{\frac \lambda {128}} \E\left[e^{-\frac {\lambda^2 k}{128} U} \right]\le \frac {384 a e^{\frac \lambda {128}}} {\lambda^2 k},
\]
ensuring $N_{R+1}>\lambda k/100$ with failure probability bounded by $c/\lambda^2 k$.
We now conclude the proof as in \cite{nguyen2022subcritical}. Once the number of infected leaves is above the threshold $\lambda k/100$, we use repeateadly Lemma~\ref{lem:metastable_step} to maintain the same amount of infected leaves for an exponentially long time and with an exponentially small failure probability.

\end{appendices}

\end{document}